\documentclass[12pt]{amsart}
\usepackage{amssymb,amsmath, enumerate, tikz, mathrsfs, inputenc, verbatim, bbm, tikz, pgffor, epsfig, mathrsfs, xcolor}
\usepackage{amsthm}
\usepackage{graphicx}

\newtheorem{thm}{Theorem}[section]
\newtheorem{prop}[thm]{Proposition}
\newtheorem{lem}[thm]{Lemma}
\newtheorem{cor}[thm]{Corollary}
\numberwithin{equation}{section} 

\theoremstyle{definition}

\newtheorem{defn}[thm]{Definition}

\theoremstyle{remark}

\newcommand{\ato}{\alpha_{21}}

\newcommand{\D}{\Delta}
\renewcommand{\H}{\mathcal H^\ast}
\newcommand{\Kh}{K_h}
\newcommand{\lo}{\ell_1}
\newcommand{\ltw}{\ell_2}
\newcommand{\lth}{\ell_3}

\newcommand{\Rh}{\mathcal R_h}

\newcommand{\s}{\mathfrak h_{z_1, z_2}}
\newcommand{\tho}{\theta_1}
\newcommand{\ttw}{\theta_2}

\newcommand{\zt}{z_3}

\renewcommand{\H}{\mathcal H}

\begin{document} 

{\allowdisplaybreaks
\title[Symmetrization of Cauchy-Like kernels]{Symmetrization of a family of Cauchy-Like kernels: \\ Global instability}
\author[Lanzani and Pramanik]{Loredana Lanzani
and Malabika Pramanik}
\address{
Dept. of Mathematics,       
Syracuse University 
Syracuse, NY 13244-1150 USA}
  \email{llanzani@syr.edu}
\address{
Dept. of Mathematics\\University of British Columbia
\\
Room 121, 1984 Mathematics Rd.
\\Vancouver, B.C. Canada   V6T 1Z2}
\email{malabika@math.ubc.ca}
\begin{abstract} 
The fundamental role of the Cauchy transform in harmonic and
complex analysis has led to many different proofs of its $L^2$ boundedness.
In particular, a famous proof of  Melnikov-Verdera \cite{MV} relies upon an iconic symmetrization identity of Melnikov \cite{M} 
 linking the universal Cauchy kernel $K_0$ to
Menger curvature.  
Analogous identities hold for the real and the imaginary parts of $K_0$ as well. Such connections have been immensely productive in the
study of singular integral operators and in geometric measure theory.
\vskip0.1in 
\noindent In this article, given any function $h: \mathbb C \rightarrow \mathbb R$,
we consider an inhomogeneous variant $K_h$ of $K_0$ which is inspired by complex
function theory. While an operator with integration kernel $K_h$ is easily
seen to be $L^2$-bounded for all $h$, the symmetrization identities for each of the real and imaginary parts of
$K_h$ show a striking lack of robustness in terms of boundedness and
positivity, two properties that were critical in \cite{MV}
 and in subsequent works by many authors. Indeed here we show that
for any continuous $h$ on $\mathbb C$, the only member of $\{K_h\}_h$ whose symmetrization
has the right properties is $K_0$! 
This global instability 
complements our
previous investigation \cite{LPr} of symmetrization identities in the restricted setting of a curve, 
where a sub-family of $\{K_h\}_h$ displays very different behaviour than its global counterparts considered here. Our methods of
proof have some overlap with techniques in recent work of Chousionis-Prat \cite{CP} and Chunaev \cite{C}.
\end{abstract}
\maketitle
\section{Introduction}
\noindent Multivariate algebraic expressions that are invariant
under permutations of the underlying variables are termed
symmetric forms. Identities involving such forms, henceforth referred to as
{\bf {symmetrization identities}}, abound in mathematics. Their appeal
lies in the physical interpretation of the various quantities that
they embody, which could be geometric, analytic or combinatorial in
nature. One such instance \cite{M} was discovered in 1995 by M. Melnikov 
in regards to
 the {\bf universal Cauchy kernel}
 \begin{equation}\label{E:01}
K_0(w, z) \ :=\ 
\frac{1}{w-z}, \qquad z, w \in \mathbb C, \; z \ne w. 
\end{equation}
This exceptionally simple identity is the source of many results of import in harmonic and complex analysis, and the starting point of this article. We begin by describing it.
\vskip0.1in
\noindent For any complex-valued function $K$ defined on a domain in $\mathbb C^2$, let $\mathtt S[K]$ denote the following {\bf{symmetrized form associated with $K$:}}
\[ \mathtt S[K] := \sum_{\sigma \in S_3} K(z_{\sigma(1)}, z_{\sigma(2)}) \overline{K(z_{\sigma(1)}, z_{\sigma(3)})}, \]
where $S_3$ denotes the group of permutations over three elements, and 
 $\mathbf z = \{z_1, z_2, z_3\}$ is any three-tuple of distinct points in $\mathbb C$ for which the above expression is meaningful. The symmetrization identity of Melnikov \cite{M} concerns $\mathtt S[K_0]$. Namely, it says 
\begin{equation}
\mathtt S[K_0] = \sum\limits_{\sigma\in S_3} \frac{1}{z_{\sigma(1)} -
    z_{\sigma(2)}} \times \frac{1}{\overline{z_{\sigma(1)} -
      z_{\sigma(3)}}} = c^2(\mathbf z), \label{E:sym-id-0p}
      \end{equation} 
 where $c(\mathbf z)$ represents the {\bf {Menger curvature}} for the points  $\{ z_1, z_2,
z_3\}$. Let us recall that the Menger curvature of three non-collinear points
 is the reciprocal of the radius of the unique circle that passes through
these three points. If the points are collinear, $c(\mathbf z)$ is taken to be
zero.  Moreover, the Menger curvature $c^2(\mathbf z)$ splits evenly between the symmetric forms of the real and imaginary parts, in the following sense:   
\[\mathtt S[K_0] = \mathtt S[\text{Re}(K_0)] + \mathtt S[\text{Im}(K_0)], \footnote{This general fact about $\mathtt S$ is proved in Section \ref{SS:Basic}} \text{ with } \] 
      \begin{align}  
\mathtt S[\text{Re}(K_0)] = \sum\limits_{\sigma\in S_3} \text{Re} \Bigl(\frac{1}{z_{\sigma(1)} -
    z_{\sigma(2)}}\Bigr) \times \text{Re}
  \Bigl(\frac{1}{{z_{\sigma(1)} - z_{\sigma(3)}}}\Bigr)\,
  &=\ \frac{1}{2}\,c^2(\mathbf z), \label{E:sym-id-0pr} \\ 
\mathtt S[\text{Im}(K_0)] = \sum\limits_{\sigma\in S_3} \mathrm{Im}\Bigl(\frac{1}{z_{\sigma(1)} -
    z_{\sigma(2)}}\Bigr)  \times
  \mathrm{Im}\Bigl(\frac{1}{{z_{\sigma(1)} -
      z_{\sigma(3)}}}\Bigr)\,  &=\  \frac{1}{2}\,c^2(\mathbf z). \label{E:sym-id-0pi}
  \end{align} 
 In 1995, Melnikov and Verdera \cite{MV} discovered that the identity \eqref{E:sym-id-0p} leads to
a new proof of famous results  by A. Calder\`on; Coifman, McIntosh and Meyer that established $L^2(\Gamma, \sigma)$-regularity of the Cauchy transform for a planar Lipschitz curve $\Gamma$ \cite{{Calderon}, {CMM}}. The new proof \cite{MV} and the subsequent 
work by Mattila, Melnikov and Verdera \cite{MMV} brought to the fore new connections between the analytic capacity and the rectifiability of the support of the reference measure: thus  a hugely prolific line of investigation began that has had
  a deep and ever-lasting impact
  on the
theory of singular integral operators and on geometric measure theory. A survey of the extensive literature would take us outside of the scope of the present work; instead we defer to the monographs \cite{P} and \cite{T} and,
for more recent progress, to \cite{{CMPT}, {CMPT2}, {CP}, {C}, {CMT}, {CMT2}}.
\vskip0.1in
%
%
%
%
%

\subsection{Objectives} The purpose of this article is to provide analogues of the global symmetrization
identities \eqref{E:sym-id-0p}, \eqref{E:sym-id-0pr} and \eqref{E:sym-id-0pi}  for a new family of integration kernels that 
arise naturally in 
complex function theory and in potential theory. Specifically, we investigate two basic features of such identities, 
\vskip0.1in 
\begin{itemize}
\item[\tt(i)] {\em{Global boundedness relative to Menger curvature}};
\vskip0.1in
and
 \vskip0.1in 
 \item[\tt(ii)] {\em{Global positivity.}} 
 \end{itemize}
 \vskip0.1in
 These properties play an important role in the proofs of many of the results mentioned above and are
enjoyed by  the symmetrized forms of each of $K_0$, $\text{Re}K_0$ and $\text{Im}K_0$ 
following \eqref{E:sym-id-0p}, \eqref{E:sym-id-0pr} and \eqref{E:sym-id-0pi}. Here and throughout, by {\bf global property} we mean a property that holds for all three-tuples $\mathbf z = \{z_1, z_2, z_3\}$ of distinct points in $\mathbb C$. 


\vskip0.1in 

\noindent In this article we consider the family of kernels
$\{K_h \}$ parametrized by functions $h: \mathbb C\to \mathbb R$, 
\begin{equation} \label{def-Kh}
K_h(w, z) := \frac{e^{i h(w)}}{w-z},\quad w,\ z\in\mathbb C,\ w\neq z.
\end{equation} 
 Setting $h \equiv 0$ (or a constant) yields the universal Cauchy kernel $K_0$ (or a constant multiple of it). On the other hand, for non-constant $h$ our kernel
   $K_h$ is significantly different in nature from $K_0$ because the latter is a homogeneous function of $w$ and $z$, while the former is not. The definition of $K_h$ is driven by complex-analytic considerations that 
 we have discussed in detail
  in the companion paper \cite{LPr}, where the focus is specialized to 
  those continuous functions $h: \mathbb C\to \mathbb R$ such that 
 $j^*K_h= K_\Gamma$ where $j: \Gamma \hookrightarrow \mathbb C$ is the inclusion map and
 \begin{equation}\label{E:K-restr}
 K_\Gamma (w, z) =\frac{1}{2\pi i}\frac{\mathbf t_\Gamma (w)}{(w-z)},\quad w, z\in \Gamma,\quad w\neq z
 \end{equation}
 for a given curve $\Gamma$ of class $C^1$ with unit tangent vector $\mathbf t_\Gamma$. By contrast here we work with the unrestricted $K_h$ given in \eqref{def-Kh} and we ask the following questions: 
 \vskip0.1in
{\em Are the symmetrization identites \eqref{E:sym-id-0pr} and \eqref{E:sym-id-0pi} preserved by the family $\{K_h\}_h$? Even if not, do the basic features {\tt(i)} and {\tt(ii)} still hold?
 }
 \vskip0.1in 
\noindent We answer both these questions in the negative. Specifically, we show that the exact analogues of the identities \eqref{E:sym-id-0pr} and \eqref{E:sym-id-0pi}  no longer hold for 
 ${\tt S}[\text{Re}K_h]$ or ${\tt S}[\text{Im}K_h]$ in the sense that neither is comparable to
 $c^2(\mathbf z)$. We do obtain identities that serve as substitutes of \eqref{E:sym-id-0pr} and \eqref{E:sym-id-0pi}, see Proposition \ref{P:symm-ids-Re-Im}. But these new identities show that
 neither  ${\tt S}[\text{Re}K_h]$ nor ${\tt S}[\text{Im}K_h]$ satisfies any of
  the basic features {\tt(i)} and {\tt(ii)}.
  More precisely,
in Theorem \ref{T:L-infty} we show that the condition ``$h \equiv$ constant on $\mathbb C$'' is {\em{equivalent}} to any one of $\mathtt S[\text{Re} K_h](\mathbf z)/c^2(\mathbf z)$ or $\mathtt S [\text{Im}(K_h](\mathbf z)/c^2(\mathbf z)$ being globally bounded: hence basic feature {\tt(i)} fails for  $\mathtt S[\text{Re} K_h]$ and $\mathtt S[\text{Im} K_h]$  with non-constant $h$;
 in fact it
 fails already
  for three-tuples confined to tri-disks, that is
 $\mathbf z = (z_1, z_2, z_3) \in \mathbb D^3$  for any disk $\mathbb D\subset \mathbb C$.
On a related theme, in Theorem \ref{T:positive} we show that the condition ``$h \equiv$ constant'' is {\em{equivalent}} to any one of $\mathtt S[\text{Re} K_h]$ or $\mathtt S[\text{Im} K_h]$ maintaining a fixed sign across $\mathbb C^3$; thus the basic feature {\tt(ii)} also fails for  $\mathtt S[\text{Re} K_h]$ and $\mathtt S[\text{Im} K_h]$ whenever $h$ is non-constant. 
\vskip0.1in

\subsection{Conclusion} For nontrivial $h$ the kernels $\{\text{Re}K_h\}$ and $\{\text{Im}K_h\}$ fail to satisfy both of the basic features {\tt(i)} and {\tt(ii)}:
%
 this is surprising because the symmetrization identity for the full universal Cauchy kernel $K_0$, see \eqref{E:sym-id-0p}, and the corresponding basic features {\tt(i)} and {\tt (ii)}, hold trivially when $K_0$ is replaced by $K_h$ 
  on account of the fact that
 $$
 \mathtt S\,[K_h](\mathbf z)  = \mathtt S\,[K_0](\mathbf z) 
 $$
for any $\mathbf z$ in $\mathbb C^3$ and  for any $h: \mathbb C\to\mathbb R$ (no continuity assumption needed).
\vskip0.1in
\noindent This dichotomy stands in contrast with symmetrized identities in the restricted context, and is best appreciated in comparison with the latter class of results. In \cite{LPr}, we study the symmetrized forms $\mathtt{S}[K_0]$, $\mathtt{S}[\text{Re}(K_0)]$ and $\mathtt{S}[\text{Im}(K_0)]$ for triples $\mathbf z = (z_1, z_2, z_3) \in \mathbb C^3$, where each $z_j$ lies on a fixed rectifiable curve $\Gamma$ with prescribed regularity. We show that a few members of $\{\text{Re}K_h\}$, namely those for which $j^*K_h = K_\Gamma$ with $K_\Gamma$  as in
\eqref{E:K-restr}, will satisfy $\Gamma$-restricted analogues of the basic features {\tt (i)} and {\tt (ii)}, while their counterparts in $\{\mathrm{Im}\Kh\}$ satisfy $\Gamma$-restricted versions of {\tt(i)} though not necessarily of {\tt(ii)}. 
\vskip0.1in
\noindent
 After completing this work, we learned that  some of the methods of proof we employ here bear similarities to techniques developed in unrelated recent work by Chousionis and Prat \cite{CP} and by Chunaev \cite{C}. For instance,
 the proofs of our representation formul\ae\,  for  $\mathtt S[\text{Re} K_h]$ or $\mathtt S[\text{Im} K_h]$, given in Proposition \ref{P:symm-ids-Re-Im} below, rely upon a certain labeling scheme of the vertices of a triangle, which is a device that also appears in \cite[Proposition 3.1]{CP} and \cite[Lemma 6]{C}; see also Figures 3 and 4 in \cite[p. 2738]{C} and the computations accompanying those diagrams.

     \subsection{Organization of this paper} 
     
     For the reader's convenience, the
      well-known background is summarized in  section \ref{S:Background}.
     In sections \ref{SS:Setup} through \ref{SS:global boundedness}
     we state  the results pertaining to Re$\Kh$ and Im$\Kh$.
      Further results pertaining to $K^*_h$ (the dual kernel of $\Kh$)  are stated in  section \ref{SS:further}.
          All proofs are collected in section \ref{S:proofs-main}.  
          
      \vskip0.1in


\section{Description of the results} \label{Statement-of-Results-section}
\subsection{Setup}\label{SS:Setup} 
\vskip0.1in  
As mentioned earlier, our kernel $K_h$ is given by 
 \begin{equation}\label{E:02}
K_h (w, z) \ :=\  
\frac{e^{i h(w)}}{w-z}
\end{equation}
where $h: \mathbb C\to \mathbb R$ is a given function. If $h$ is a constant function, then $K_h$ is a constant multiple of $K_0$ given in \eqref{E:01}. 

\subsection{Symmetrization identites for arbitrary $h$.}\label{SS: symmetrization identities} 
It is immediate to see that
 \begin{equation}\label{E:same}
 \mathtt S\,[K_h](\mathbf z) \, =\, 
\mathtt  S[K_0] (\mathbf z) = c^2(\mathbf z)\
 \end{equation}
 for any three-tuple of distinct points
 and for any $h:\mathbb C\to\mathbb R$ (no continuity assumption needed here).
\vskip0.1in
\noindent The first natural question that presents itself is whether the phenomenon \eqref{E:same} is inherited by the real and the imaginary parts of $\Kh$.
 On account of \eqref{E:sym-id-0pr} and \eqref{E:sym-id-0pi} this amounts to asking
  whether the identities
\begin{equation}\label{E:sym-id-0ph}
 \begin{aligned}
 \mathtt S\,[\text{Re}K_h](\mathbf z) \ &=\, \mathtt S\,[\text{Re}K_0](\mathbf z)
 \,=\,\frac12\, c^2(\mathbf z) ,\quad \text{and}
 \\& \\
  \mathtt S\,[\mathrm{Im}K_h](\mathbf z) \ &=\, \mathtt S\,[\mathrm{Im}K_0](\mathbf z)
  \,=\,\frac12\, c^2(\mathbf z)
\end{aligned}
  \end{equation}
  are true for all (or for some) non-constant $h:\mathbb C\to \mathbb R$. We answer this question in the negative. 
\vskip0.1in 
\noindent  In order to state the precise result
 we adopt a
 specific 
  labeling scheme for three-tuples of non-collinear points. 
\begin{defn}\label{D:admissible}
We say that an ordered three-tuple of non-collinear points
$(a, b, c)$ is arranged in {\em admissible order} (or is {\em admissible}, for short) if {\tt(a)}
the orthogonal projection of $c$ onto the line determined by $a$ and $b$ falls in the interior of the line segment joining $a$ and $b$, and {\tt(b)} the triangle $\Delta (a, b, c)$ has positive counterclockwise orientation.
\end{defn}
\noindent We will show in section \ref{SS:LabelingScheme} that any three-tuple of non-collinear points
has at least one admissible ordering.
\begin{prop}\label{P:symm-ids-Re-Im}
For any non-constant $h:\mathbb C\to \mathbb R$ and for any three-tuple $\mathbf z$
 of non-collinear points in $\mathbb C$ we have
 \item[]\quad 
 \begin{equation}\label{E:new-symm-Re}
    \mathtt S\,[\mathrm{Re}K_h](\mathbf z) 
  \displaystyle{
  =\ c^2(\mathbf z)\left(\frac12\, + \, \Rh (\mathbf z)\right)}
 \end{equation}
\begin{equation}\label{E:new-symm-Im}
  \mathtt S\,[\mathrm{Im}K_h](\mathbf z) 
  \displaystyle{
 =\ c^2(\mathbf z)\left(\frac12\, - \, \Rh (\mathbf z)\right)}
 \end{equation}
 where $\Rh $ is non-constant and invariant under the permutations of the elements of $\mathbf z$. If
 $\mathbf z= (z_1, z_2, z_3)$ is admissible then $\Rh (\mathbf z)$ is represented as follows:

 \begin{equation}\label{E:RH}
\begin{aligned}
\Rh (\mathbf z) &=\frac{2\,\ell_1\ell_2\ell_3}{\left(4\mathrm{Area}\,\Delta(\mathbf z)\right)^2}\, \times 
\Bigl[
\ell_1\cos(\s(z_1) -\tho) + 
\\ 
&\hskip1in + 
\ell_2\cos (\s(z_2) +\ttw)
-\ell_3\cos(\s(z_3) +\ttw -\tho)
\Bigr]\, .
\end{aligned} 
\end{equation} 
Here $\theta_j$ denotes the angle at $z_j$, and $\ell_j$ denotes the length of the side opposite to $z_j$ in $\Delta(\mathbf z)$.
 Also, we have set
 $$\s(z) := 2h(z) -2\ato\, ,\quad z\in\mathbb C,
  $$
where $\ato$ is the principal argument of $z_2-z_1$ (in an arbitrarily fixed coordinate system for $\mathbb R^2$). 
 \end{prop}
\vskip0.1in
 \noindent {\bf Remarks (A).}
 \begin{enumerate}[1.]
\item  \label{R:A} A discussion of the invariance properties of these symmetrized forms is in order.  While $\mathtt S[K_0](\mathbf z)$ (and thus $\mathtt S[K_h](\mathbf z)$), $\mathtt S\,[\text{Re}K_0](\mathbf z)$ and $\mathtt S\,[\mathrm{Im}K_0](\mathbf z)$) are clearly invariant under rotations and translations of 
$\mathbb R^2$, this is no longer the case for 
 $\mathtt S\,[\text{Re}K_h](\mathbf z)$ or $\mathtt S\,[\mathrm{Im}K_h](\mathbf z)$. 
 Yet, \eqref{E:new-symm-Re} and \eqref{E:new-symm-Im} show that each of
these symmetrized forms
 contains a term that is invariant under rotations and translations of $\mathbb R^2$. 
 \vskip0.05in
 \item The well-posedness of the righthand sides of \eqref{E:new-symm-Re} and \eqref{E:new-symm-Im}
  for three-tuples of distinct points, as opposed to non-collinear points,  is addressed in section \ref{SS:Symm-ids-pf}, see Remark {(B)} in that section.
 \end{enumerate}

\subsection{Boundedness and positivity: global results for arbitrary $h$.}
\label{SS:global boundedness}  
We next ask whether either 
 $\mathtt S\,[\text{Re}K_h](\mathbf z)$ or $\mathtt S\,[\mathrm{Im}K_h](\mathbf z)$ 
obeys the basic global features {\tt(i)} and {\tt(ii)}.
 The answer is negative in all instances: we prove below that $\Rh$ is either constant (in fact zero) or unbounded!

 \begin{thm}\label{T:L-infty}
 Suppose that $h:\mathbb C\to \mathbb R$ is continuous.  The following are equivalent:
\begin{itemize}
\item[]
\item[{\tt (i)}]  \ \ 
There is  a constant $C<\infty$, possibly depending on $h$, such that $$|\Rh (\mathbf z)|\leq C$$
\ \ for any  three-tuple $\mathbf z=\{z_1, z_2, z_3\}$ of non-collinear points in $\mathbb C$.

\item[]
\item[{\tt (ii)}] \quad 
$\Rh (\mathbf z) \, =\, 0$ \, for any three-tuple
of non-collinear points in $\mathbb C$.
\item[]
\item[{\tt (iii)}] \quad
 $h$ is constant.
 \item[]
\end{itemize}
\end{thm}
\noindent The proof will in fact show the following stronger conclusion: 
 if we only assume that $h$ is continuous on a disc $\mathbb D_0\subset\mathbb C$, then the boundedness of 
 $\Rh (\mathbf z)$ for any three-tuple $\mathbf z$ of points in $\mathbb D_0$ is equivalent to $h$ being constant on $\mathbb D_0$.
\vskip0.1in

\noindent An immediate consequence of Theorem \ref{T:L-infty} is
that  no analog 
of the positivity condition {\tt(ii)} can hold simultaneously for  $\mathtt S\,[\text{Re}K_h]$ and 
$\mathtt S\,[\mathrm{Im}K_h]$ unless $h=$ const. Specifically, we have
\begin{cor}\label{C:L-infty}
 Suppose that $h:\mathbb C\to \mathbb R$ is continuous. Then
$$
h\ \text{is\ constant}\quad \iff\quad
 \left(\frac12 -\Rh(\mathbf z)\right)\!\! \left(\frac12 +\Rh(\mathbf z)\right) > 0
$$
for any three-tuple of  non-collinear points in $\mathbb C$.
\end{cor}
\noindent In fact more is true.
\begin{thm}\label{T:positive}
 Suppose that $h:\mathbb C\to \mathbb R$ is continuous. 
 
 \vskip0.1in
\begin{itemize}
\item[{\tt (a)}]  \ \  If $h$ is constant, then
$$\displaystyle{\frac12 +\Rh (\mathbf z) > 0}$$
\quad for any  three-tuple
 of non-collinear points.
\item[]
\vskip0.1in
\item[{\tt (b)}] \quad If
 $h$ is not constant, then the function
 $$
 \mathbf z\ \mapsto \frac12 +\Rh (\mathbf z) 
 $$
 \quad
 changes sign.  That is, there exist two three-tuples of non-collinear points 
 
 \ $\mathbf z$ and $\mathbf z'$ such that
 $$
 \frac12 +\Rh (\mathbf z)\,  >\, 0\quad \text{and}\quad  \frac12 +\Rh (\mathbf z')\,  <\, 0.
 $$
\end{itemize}
Furthermore, 
{\tt (a)}
and {\tt (b)} are also true with $\displaystyle{\frac12 - \Rh}$ in place of $\displaystyle{\frac12 + \Rh}$.
\end{thm}
\vskip0.1in
\begin{cor}\label{C:ReKh-positive}
 Suppose that $h:\mathbb C\to \mathbb R$ is continuous. Then
 \begin{itemize}
 \item[{\tt{(a)}}] $h$ is constant \quad $\iff\quad \frac12 +\Rh(\mathbf z) > 0$
 \vskip0.05in
\item[] for all three-tuples
 of  non-collinear points.
\vskip0.1in
\item[{\tt{(b)}}] $h$ is constant \quad $\iff\quad \frac12 -\Rh(\mathbf z) > 0$
\vskip0.05in
\item[] for all three-tuples
 of  non-collinear points.
 \end{itemize}
\end{cor}

\vskip0.1in

\subsection{Further results}\label{SS:further}
\noindent  As is well-known, the Cauchy transform is essentially
 self-adjoint. In fact if $\mathfrak C_0^{\ast}$ denotes the formal $L^2$-adjoint
 of $\mathfrak C_0$, then the ``dual'' kernel of $K_0$ (which is the
 kernel for $\mathfrak C_0^{\ast}$) is

 $$
   K_0^\ast(w, z) = -
 \overline{K_0(w, z)},
$$
 giving
  \begin{equation}\label{E:Ctr-self-adj}
 \mathtt S\,[K_0^\ast](\mathbf z)\ =\  \mathtt S\,[K_0](\mathbf z)\, .
  \end{equation} 
 \vskip0.05in
\noindent Thus 
 \begin{equation}\label{E:Ctr-self-adj-c}
 \mathtt S\,[K_0^\ast](\mathbf z) = c^2(\mathbf z)
  \end{equation} 
and the symmetrization estimates for 
$K_0^\ast$
  are synonymous with those for
  $K_0$.
 
\vskip0.1in
\noindent We define the dual kernel of $K_{h} (w, z)$ as 
\begin{equation}  K_h^{\ast} (w, z) = \overline{K_h(z, w)}. \label{def-Kh-star} \end{equation}  
Thus 
\[ K_h^{\ast} (w, z) = \frac{e^{-i h(z)}}{\overline{z} - \overline{w}}. \]
As the name suggests, if $K_{h}$ is the integration kernel of an integral operator, then $K_{h}^{\ast}$ is the kernel of its formal adjoint.  
\vskip0.1in
\noindent In great contrast with \eqref{E:Ctr-self-adj} and \eqref{E:Ctr-self-adj-c}, for non-constant $h$ the symmetrization identities and global estimates for $K_h^{\ast}$ turn out to be very different from those for $K_h$. While not directly related to the sequel \cite{LPr} of this paper, these results are of independent interest and
we state them below.
\vskip0.1in
\begin{prop}\label{P:symm-id-K-star} For any non-constant $h:\mathbb C\to \mathbb R$ and for any three-tuple $\mathbf z$ of non-collinear points
 in $\mathbb C$ we have
\begin{equation}\label{E:sym-adj}
  \mathtt S\,[K_h^\ast](\mathbf z) \ =\, 
c^2(\mathbf z)\, \H(\mathbf z)\,
\end{equation}
where $\H(\mathbf z)$ is a non-constant function of $\mathbf z$ that is 
  invariant under the permutations of the elements of $\mathbf z$. In particular, if $\mathbf z =(z_1, z_2, z_3)$ is admissible then
$\H (\mathbf z)$ has the following
 representation.
\begin{equation} \label{E:H}
\begin{aligned}
\H (\mathbf z) &=\frac{2\,\ell_1\ell_2\ell_3}{\left(4\mathrm{Area}\,\Delta(\mathbf z)\right)^2}\, \times 
\Bigl[
\ell_1\cos(h(z_2) - h(z_3) +\theta_1) + \\ 
&\hskip1in + \ell_2\cos (h(z_1)-h(z_3) - \theta_2)
+\ell_3\cos(h(z_1) -h(z_2) +\theta_3)
\Bigr]\, .
\end{aligned} 
\end{equation} 
Here $\theta_j$ and $\ell_j$ are as in the statement of Proposition \ref{P:symm-ids-Re-Im}. 
\end{prop}

\vskip0.1in

 \begin{thm}\label{T:L-infty-H}
 Suppose that $h:\mathbb C\to \mathbb R$ is continuous.  The following are equivalent:
\begin{itemize}
\item[]
\item[{\tt (i)}]  \ \ 
There is  a constant $C<\infty$, possibly depending on $h$, such that $$|\H (\mathbf z)|\leq C$$
\ \ for any  three-tuple $\mathbf z=\{z_1, z_2, z_3\}$ of non-collinear points in $\mathbb C$.

\item[]
\item[{\tt (ii)}] \quad 
$\H (\mathbf z) \, =\, 1$ \, for any three-tuple
of non-collinear points in $\mathbb C$.
\item[]
\item[{\tt (iii)}] \quad
 $h$ is constant.
 \item[]
\end{itemize}
\end{thm}

\section{Background}\label{S:Background}
\noindent Here we list without proof a few elementary computational tools and basic facts.

\subsection{Basic properties of symmetrized forms} \label{SS:Basic}
Recall that the  symmetrized form  considered in \cite{LPr} is
$$
\mathtt S\,[K](\mathbf z) := \sum\limits_{\sigma\in S_3}
  K(z_{\sigma(1)},\, z_{\sigma(2)})\, 
  \overline{K(z_{\sigma(1)},\, z_{\sigma(3)})}
$$
where $S_3$ is the set of all permutations of $\{1, 2, 3\}$ and  $\mathbf z$ is any three-tuple of points in $\mathbb C$ for which the above expression is meaningful.
\begin{itemize}
\vskip0.1in
\item It is easy to see that the above can be equivalently expressed as
\begin{equation}\label{E:Symm-simpl1} 
\mathtt S\,[K](\mathbf z) =   2\,\sum\limits_{\sigma\in S'_3}
  \mathrm{Re}\big(K(z_{\sigma(1)},\, z_{\sigma(2)})\, 
  \overline{K(z_{\sigma(1)},\, z_{\sigma(3)})}\big)
  \end{equation}
  where $S'_3 =\{123, 213, 312\}$. 
  \vskip0.1in
  \item It follows that  $\mathtt S\,[K](\mathbf z)$ is real-valued and that
  \begin{equation}\label{E:Symm-cplx-conj} 
 \mathtt S\,[\overline{K}](\mathbf z)\,=\, \overline{\mathtt S\,[K](\mathbf z)}\,=\, \mathtt S\,[K](\mathbf z)
   \end{equation}
  \vskip0.1in
  \item The symmetrized form of $K(w, z)$ can also be expressed as
  \begin{equation}\label{E:Symm-simpl2}
  \mathtt S\,[K](\mathbf z)=
  2\,\sum\limits_{\stackrel{j}{k<l}}
  \mathrm{Re}\big( K(z_j,\, z_k)\, 
  \overline{K(z_j,\, z_\ell)}\big)\, ,
\end{equation} 
where it is understood that $j\in\{1, 2, 3\}$ and for each fixed $j$, the remaining labels $k, \ell\in\{1, 2, 3\}$ are displayed so that $k<\ell$. 
\vskip0.1in
\item In particular, if $A(w, z)$ is {\em real-valued}, then we have
 \begin{equation}\label{E:Symm-simpl3}
  \mathtt S\,[A](\mathbf z)=
  2\sum\limits_{\stackrel{j}{k<l}}
  A(z_j,\, z_k)\, 
  A(z_j,\, z_\ell)
\end{equation} 
\vskip0.1in
\item 
 The operation $K\mapsto\mathtt S[K]$ is not linear: 
 $$\mathtt S[K+H]\neq \mathtt S[K] + \mathtt S[H]\quad \text{and}\quad \mathtt S[aK]\neq a\mathtt S[K].$$ However 
if the kernels $A$ and $B$ are {\em real-valued} then
$$
\mathtt S[A\pm iB] = \mathtt S[A]\ +\ \mathtt S[B]\, .
$$
Thus
\begin{equation}\label{E:Symm-real-im}
\mathtt S[K] = \mathtt S[\text{Re}K]\ +\ \mathtt S[\mathrm{Im}K].
\end{equation}
\vskip0.1in
\item For the kernels $\{K_h\}_h$ defined in \cite{LPr},
 whenever convenient we will write
 \begin{equation}\label{E:K-simpl}
\Kh (w, z) \ =\  
\frac{e^{i h(w)}(\bar w-\bar z)}{|w-z|^2}
\end{equation}
and in particular
\begin{equation}\label{E:ReK-simpl}
\mathrm{Re}\,\Kh (w, z) = \frac{\mathrm{Re}\left(e^{-i h(w)}(w-z)\right)}{|w-z|^2}
\end{equation}
with a corresponding identity for the imaginary part.
\vskip0.1in
\item Finally, we point out that for the kernels $K_h$ we have
\begin{equation}\label{E:K-simpl1}
 \mathtt S\,[K_h\,](\mathbf z) \, =\, \mathtt S\,[K_{h+c}\,](\mathbf z)
\quad\text{for any constant}\ c\in\mathbb R\, .
\end{equation}
\vskip0.1in
\end{itemize}
\subsection{A brief review of Menger curvature} The Menger curvature associated to any
three element set $\{a, b, c\}$ of distinct points in $\mathbb C$, 
denoted $c(\mathbf z)$, is the reciprocal of the radius of the circle passing through 
those  points (with the understanding that if the points are  collinear, then
 $c(\mathbf z)=0$). Suppose now that the three points are not collinear and consider the triangle $\Delta (\mathbf z)$ with vertices $\{a, b, c\}$, which we describe as follows:  for each $j\in\{a, b, c\}$ we
   denote the angle at $j$ by $\theta_j$, while $\ell_j$ denotes the length of the side opposite to $j$, that is: $\ell_j= |z_l-z_k|$ where
 $\{ j, k, l\} = \{a, b, c\}$. See Figure \ref{fig:Figure1} below.
  \begin{figure}[h!]\label{Figure 1}
    \centering

  \begin{tikzpicture}

    \draw [thin] (0,0)--(4,4);
    \draw [thin] (4,4) -- (6,3);
    \draw [thin] (0,0) -- (6,3);

    \draw (0.8,0.8) arc (45:20:0.8cm);
    \draw (3.7,3.7) arc (-135:-24.5:0.4cm);
    \draw (5.6,5.6/2) arc (-160:-205:0.5cm);

    \draw (0,0) node[below] {$c$};
    \draw (4,4) node[above] {$a$};
    \draw (6,3) node[right] {$b$};

    \draw (2.5,2.5) node[above left] {$\ell_b$};
    \draw (5,3.5) node[above right] {$\ell_c$};
    \draw (3,1.5) node[below right] {$\ell_a$};

    \draw (0.9,0.9) node[right] {$\theta_c$};
    \draw (5.6,3) node[left] {$\theta_b$};
    \draw (4,3.6) node[below] {$\theta_a$};

 \end{tikzpicture}
\caption{{\em Labeling system for vertices, side-lengths and angles of a triangle.}}
\label{fig:Figure1}
  \end{figure}
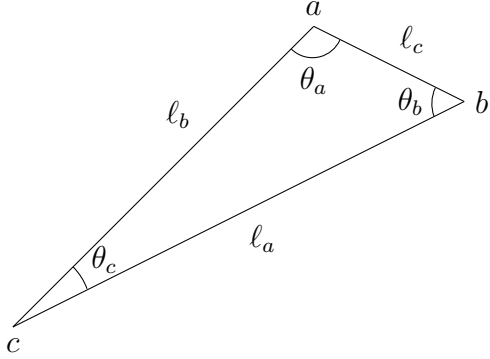

\vskip0.05in  
 \noindent We recall a few basic identities that relate $c(\mathbf z)$ and $\Delta (\mathbf z)$, see e.g. \cite{Verd}:

 \begin{equation}\label{E:Msin-1}
 c(\mathbf z)\,=\, 2\,\frac{\sin\theta_j}{\ell_j}\, , \quad j\in\{a, b, c\},\quad \text{so\ in \ particular}
 \end{equation}
 
  \begin{equation}\label{E:Msin-2}
\text{(Law of sines)} \quad \frac{\sin \theta_a}{\ell_a} = \frac{\sin \theta_b}{\ell_b} = \frac{\sin \theta_c}{\ell_c};
 \end{equation}
 \begin{equation}\label{E:Mcos-1}
 2\,\frac{\cos\theta_j}{\ell_k\ell_l}\, =\,  c^2(\mathbf z)\, 
 \frac{\sin 2\theta_j}{4\sin\theta_a\sin\theta_b\sin\theta_c}\, ,\qquad j\in\{a, b, c\}
 \end{equation}
 with the understanding that for any given $j=a, b, c$, one takes $l$ and $k$ to be the other two labels in $\{a, b, c\}$,
  and it follows from \eqref{E:Mcos-1} that
 
 \begin{equation}\label{E:Mcos-2}
 \frac{\cos\theta_a}{\ell_b\ell_c} +
 \frac{\cos\theta_b}{\ell_a\ell_c} +
 \frac{\cos\theta_c}{\ell_a\ell_b}
 \ = \
 \frac{c^2(\mathbf z)}{2}\, ,
 \end{equation}
 because
  $
  \sum\limits_{j\in\{a, b, c\}}\!\!\sin 2\theta_j= 4\sin \theta_a\sin\theta_b\sin\theta_c$, 
  see again Figure \ref{fig:Figure1}. We also have
    \begin{equation}\label{E:Marea-1}
 c(\mathbf z)\,=\, \frac{4\,\text{Area}(\Delta (\mathbf z))}{\ell_a\ell_b\ell_c},
 \end{equation}
 which immediately leads us to
 \begin{equation}\label{E:aux-one}
c^2(\mathbf z)\frac{\ell_a\ell_b\ell_c}{(4\,\mathrm{Area}\,\Delta (\mathbf z))^2}\ =\ \frac{1}{\ell_a\ell_b\ell_c}\, .
\end{equation}
\vskip0.1in

  \subsection{Labeling scheme for $\D (\mathbf z)$}\label{SS:LabelingScheme} In representing $\Rh(\mathbf z)$, see 
  \cite[Proposition 5.2]{LPr}, it will be convenient to re-label the vertices $\{a, b, c\}$
  as, say, $\{z_1, z_2, z_3\}$ where $z_3$ is any vertex whose orthogonal projection onto the line
  determined by the two other vertices falls {\em into} the side of $\D(\mathbf z)$ that is opposite to $z_3$ and, furthermore, $z_1$ and $z_2$ are labeled so that the ordered three-tuple $(z_1, z_2, z_3)$ has positive (counterclockwise) orientation. We also relabel the angles $\theta_j$ and the sides $\ell_j$ accordingly. We recall from 
\cite{LPr}   that such labels are called {\em admissible}, see Figure 2 below. All subsequent formulae appearing in this section refer to admissible ordered three-tuples $(z_1, z_2, z_3)$.
  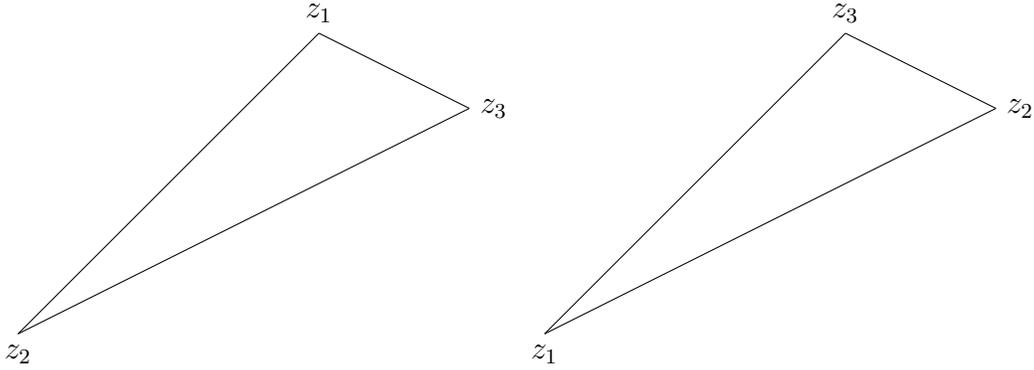
\begin{figure}[h!]\label{Figure 2}
     \centering

       \begin{tikzpicture}

    \draw [thin] (0,0)--(4,4);
    \draw [thin] (4,4) -- (6,3);
    \draw [thin] (0,0) -- (6,3);

    \draw (0,0) node[below] {$z_2$};
    \draw (4,4) node[above] {$z_1$};
    \draw (6,3) node[right] {$z_3$};

    \draw [thin] (7,0)--(11,4);
    \draw [thin] (11,4) -- (13,3);
    \draw [thin] (7,0) -- (13,3);

    \draw (7,0) node[below] {$z_1$};
    \draw (11,4) node[above] {$z_3$};
    \draw (13,3) node[right] {$z_2$};

 \end{tikzpicture}
\caption{{\em A triangle in non-admissible
form (left), and in admissible form (right).}}
\label{fig:Figure2}
\end{figure} 
   
\noindent Since the orthogonal projection of $z_3$ onto the line
  determined by $z_1$ and $z_2$ occurs at a point $z_4$ that lies {\em into} the side of $\D(\mathbf z)$ that is opposite to $z_3$,
   we have that
  $$
  z_4\ =\ z_1 + (z_2-z_1)\beta\ =\ z_2 + (z_1-z_2)(1-\beta)\quad \text{for some}\quad
  0<\beta<1.
  $$
  From these it follows that
  $$
  z_3-z_4\ =\ i\,\beta\tan\theta_1\, (z_2-z_1)\ =\ 
  i\, (1-\beta)\tan\theta_2\, (z_2-z_1)\, ,
  $$
  which in turn grant
  \begin{equation}\label{E:trig-id-1}
  z_2-z_3 = \frac{(1-\beta)}{\cos \theta_2}\,e^{-i\theta_2}(z_2-z_1)
  \end{equation}
  and
  \begin{equation}\label{E:trig-id-2}
  z_3-z_1 =\frac{\beta}{\cos\theta_1}\,e^{i\theta_1}(z_2-z_1)\, .
  \end{equation}
Finally, we recall for future reference that
\begin{equation}\label{E:elem-trig}
\ell_2\cos\theta_1 = |z_1-z_4|=\beta\,|z_2-z_1|=\beta\ell_3\, ,
\qquad\text{and}\qquad
\ell_1\cos\theta_2 = (1-\beta)\ell_3\, .
\end{equation}
\vskip0.1in
 \section{Proofs}\label{S:proofs-main}

 \subsection{Proof of  Proposition \ref{P:symm-ids-Re-Im}}
  \label{SS:Symm-ids-pf}
On account of \eqref{E:Symm-real-im} we only need to prove
 the symmetrization identity for $\mathrm{Re}\,\Kh$. To this end we start with
\begin{equation}\label{E:ReImSym1}
{\tt S}[\,\mbox{Re}\,K_h\,](\mathbf z) =
\frac{2}{\lo^2\,\ltw^2\,\lth^2}
\!
\sum\limits
_{\stackrel{j}{k<l}}
 \ell^2_j
\,
\mbox{Re}\,\big(e^{-ih(z_j)}(z_j-z_k)\big)\,
\mbox{Re}\,\big(e^{-ih(z_j)}(z_j-z_l)\big)\, .
\end{equation}
  We expand the sum and again use \eqref{E:trig-id-1} and \eqref{E:trig-id-2} to express $(z_j-z_k)$ and
  $(z_j-z_\ell)$ in terms of $(z_2-z_1)$. Writing
  $$
  z_2-z_1\ =\ \ell_3\, e^{i\alpha_{21}}
  $$
   we are led to
  $$
  {\tt S}[\,\mbox{Re}\,K_h\,](\mathbf z) =
   \frac{\beta}{\ell_2^2\cos\theta_1}\, \Big[2\cos \!\big(h(z_1)-\alpha_{21}\big)
\cos\!\big(h(z_1)-\alpha_{21}-\theta_1\big)\Big]\ +\ 
  $$
  $$
   \frac{(1-\beta)}{\ell_1^2\cos\theta_2}\, \Big[2\cos\!\big(h(z_2)-\alpha_{21}\big)
\cos\!\big(h(z_2)-\alpha_{21}+\theta_2\big)\Big]\ +\
$$
$$
 -\ \ell_3^2\,\frac{\beta}{\ell_2^2\cos\theta_1}\, \frac{(1-\beta)}{\ell_1^2\cos\theta_2}\, \Big[2\cos\!\big(h(z_3)-\alpha_{21}-\theta_1\big)
\cos\!\big(h(z_3)-\alpha_{21}+\theta_2\big)\Big]\,.
  $$
  \vskip0.1in
 \noindent Applying the formula $$2\cos\gamma\cos\lambda = \cos(\gamma+\lambda) +
 \cos(\gamma-\lambda)$$
 to each of the three summands (for appropriate choices of $\gamma$ and $\lambda$)
 and recalling our definition $\s (z) := 2(h(z)-\alpha_{21}),$ we obtain
$$
  {\tt S}[\,\mbox{Re}\,K_h\,](\mathbf z) =
  $$
  $$=\
   \frac{\beta}{\ell_2^2\cos\theta_1}\big[
   \cos \!\big(\s(z_1)-\theta_1\big)+
\cos\theta_1
\big]\ +\ 
   \frac{(1-\beta)}{\ell_1^2\cos\theta_2}\big[
   \cos\!\big(\s(z_2)+\theta_2\big) +
\cos\theta_2\big]\ +
$$
$$
 -\ \ell_3^2\,\frac{\beta}{\ell_2^2\cos\theta_1}\, \frac{(1-\beta)}{\ell_1^2\cos\theta_2}
 \big[
 \cos\!\big(\s(z_3)+\theta_2-\theta_1\big) -\cos\theta_3\big]\, .
  $$
  \vskip0.1in
 \noindent On account of \eqref{E:elem-trig},
the expression above is reduced to
  $$
   \frac{1}{\ell_2\,\ell_3}\big[
   \cos \!\big(\s(z_1)-\theta_1\big)+
\cos\theta_1
\big]\ +\ 
   \frac{1}{\ell_1\,\ell_3}\big[
   \cos\!\big(\s(z_2)+\theta_2\big) +
\cos\theta_2\big]\ +
$$
$$
 -\ \frac{1}{\ell_2\,\ell_1}
 \big[
 \cos\!\big(\s(z_3)+\theta_2-\theta_1\big) -\cos\theta_3\big]\, .
  $$
Applying \eqref{E:Mcos-2} we are led to
$$
  {\tt S}[\,\mbox{Re}\,K_h\,](\mathbf z) = \frac{c^2(\mathbf z)}{2}\ +
  $$
  $$+\
  \frac{1}{\ell_1\ell_2\ell_3}\Big[\ell_1\cos \!\big(\s(z_1)-\theta_1\big)+  \ell_2\cos\!\big(\s(z_2)+\theta_2\big) -
  \ell_3\cos\!\big(\s(z_3)+\theta_2-\theta_1\big)
  \Big]\, .
$$
\vskip0.05in
\noindent The symmetrization identity for $\mathrm{Re} K$ now follows from \eqref{E:aux-one}.
The proof is concluded.
\vskip0.1in
\noindent  A word on the well-posedness of the definition of $\Rh (\mathbf z)$ is in order. If the triangle
   $\D(\mathbf z)$ has an obtuse or right angle, then $\Rh (\mathbf z)$ is unambiguously defined 
   in the sense that there is a unique admissible form of $\mathbf z$ (there is a unique permutation of $\{z_1, z_2, z_3\}$ that gives the admissible form of $\mathbf z$).
    On the other hand, if $\D(\mathbf z)$ is an acute-angle triangle (all three angles in $\D(\mathbf z)$ are acute) then there are three distinct admissible orderings of $\mathbf z$ because $z_3$ can be assigned to be {\em any one} of
   the three vertices $a$, $b$ or $c$. Correspondingly there are three
   formulations of $\Rh (\mathbf z)$: these, however, must be identical to one another in view of the invariance of  $\Rh (\mathbf z)$
   under the permutations of $\{z_1, z_2, z_3\}$.
      Alternatively, 
  one can directly verify that the three admissible forms of $\mathbf z$ lead to the same representation for 
  $\Rh (\mathbf z)$ by invoking
   the following lemma, whose proof is omitted:
   \begin{lem}\label{L:trig-2} Let $\D(\mathbf z)$ be any triangle with vertices $\{a, b, c\}$, see Figure 1. For any $j, k\in \{a, b, c\}$ let $\alpha_{jk}\in [0, 2\pi)$ denote the argument
   of $j-k$ (in an arbitrarily fixed coordinate system for $\mathbb R^2$). Then, with the notation of Figure 1, we have
   \begin{equation}\label{E:trig-abc}
\alpha_{ac} = \alpha_{ba} + \theta_a + \pi \ \ mod\, (2\pi);\quad
\alpha_{bc} = \alpha_{ba} - \theta_b + \pi \ \ mod\, (2\pi)\, .
   \end{equation}
   \end{lem}
\vskip0.1in
\noindent {\bf Remark (B).}
Note that $\Rh (\mathbf z)$ is meaningful only when the points in the three-tuple $\mathbf z=\{z_1, z_2, z_3\}$ are non-collinear (so that
Area$(\Delta (\mathbf z))\neq 0$). 
However  the set of non-collinear three-tuples in 
$\mathbb C$ viewed as a subset of $\mathbb C^3$ has full Lebesgue measure
because the condition that Area$(\Delta (\mathbf z))= 0$ is equivalent to
\begin{equation}\label{E:algebraic}
\vec{u}\times\vec{v}=\vec{0}
\end{equation}
with $\times$ denoting the cross product in $\mathbb R^2$ of the vectors 
$\vec{u} := z_2-z_1$ and $\vec{v} := z_3-z_1$.
  Since
 \eqref{E:algebraic} is a quadratic equation in the variables $(z_1$, $z_2$, $z_3)$,  its solution set 
 has zero Lebesgue measure in $\mathbb C^3$ (it is an algebraic subvariety of $\mathbb C^3$);  
 on the other hand, the product
$$c^2(\mathbf z )\Rh (\mathbf z)\, $$
 is meaningful as soon as the points $\{z_1, z_2, z_3\}$ are distinct from one another because, on account  \eqref{E:aux-one}, it equals
$$
\frac{1}{\ell_1\ell_2\ell_3}
\,\times\, \mathrm{(a\ continuous\ function\ of}\  \mathbf z\, \mathrm)\, ,
$$
and $\ell_j\neq 0$ for each $j= 1, 2, 3$  (since the $z_j$'s are distinct).

\subsection{Proof of 
Theorem \ref{T:L-infty}} 
 Note that the implication {\tt (ii)} $\Rightarrow$ {\tt (i)} is trivial, thus it suffices to show
that
 \medskip
 
 \centerline{{\tt (iii)} $\Rightarrow$ {\tt (ii)}\quad  and\quad   {\tt (i)} $\Rightarrow$ {\tt (iii)}.}
\vskip0.07in

\noindent We begin by proving that {\tt (iii)} $\Rightarrow$ {\tt (ii)}. We claim that

\medskip

\centerline{ {\tt (iii)} $\ \Rightarrow\  
c^2(\mathbf z)\Rh (\mathbf z) =0$ 
(which immediately implies {\tt (ii)}).}
\medskip
\noindent To see this, let $c_0$ denote the assumed constant value of $h(z)$; we use the short-hand notation
$$
\alpha = 2c_0 - 2\alpha_{2 1}
$$
(note that $\alpha = \alpha (z_1, z_2)$ because $\alpha_{21}$ is a function of $z_1$ and $z_2$).
 Combining  Proposition \ref{P:symm-ids-Re-Im}
  the basic identity \eqref{E:Mcos-2} we see that  
$$
c^2(\mathbf z)\, \Rh(\mathbf z) =
 \frac{1}{\ltw\lth}\cos\!\big(\alpha -\theta_1\big) +
\frac{1}{\lo\lth}\cos\!\big(\alpha +\theta_2\big) -
\frac{1}{\lo\ltw}\cos\!\big(\alpha+ \theta_2-\theta_1\big)\,
$$
and thus
$$
c^2(\mathbf z)\, \Rh(\mathbf z) = C (\mathbf z)  \cos\alpha \ +\ D (\mathbf z) \sin\alpha\, 
$$
where
$$
C (\mathbf z) := 
\frac{\cos\theta_1}{\ltw\lth} +
\frac{\cos\theta_2}{\lo\lth} -
\frac{1}{\lo\ltw}\cos(\theta_2-\theta_1)
$$
and
$$
D (\mathbf z) := 
\frac{\sin\theta_1}{\ltw\lth} -
\frac{\sin\theta_2}{\lo\lth} +
\frac{1}{\lo\ltw}\sin(\theta_2-\theta_1).
$$
We claim that
$$
C (\mathbf z)=  0\, ,\quad \text{and}\quad D (\mathbf z)=  0\, .
$$
Indeed, it follows from \eqref{E:Mcos-2}, \eqref{E:Msin-1} and the identity 
$\cos\theta_3 = -\cos(\theta_1+\theta_2)$,  that
$$
C (\mathbf z) = \frac12\, c^2(\mathbf z)\, - \left(\frac{\cos\theta_3 + \cos (\theta_2-\theta_1)}{\lo\,\ltw}\right) =
$$
$$
\quad = \frac12 c^2(\mathbf z)\ -\ 2\frac{\sin\theta_1}{\lo}\frac{\sin\theta_2}{\ltw} =\
\frac12\, c^2(\mathbf z)\ -\ \frac12\, c^2(\mathbf z)\ =\ 0.
$$
To deal with $D(\mathbf z)$ we invoke \eqref{E:Msin-2} and the identity $\sin\theta_3 = \sin(\theta_1+\theta_2)$, which lead to
$$
D (\mathbf z) = \frac{1}{\lo\ltw\sin\theta_3}\,
\big(\sin^2\theta_1-\sin^2\theta_2 + \sin(\theta_2+\theta_1)\sin(\theta_2-\theta_1)\big)\ =\ 0.
$$
The proof of the implication: {\tt (iii.)} $\Rightarrow$ {\tt (ii.)} is concluded.
\bigskip

\noindent Next we prove that {\tt (i)}$\Rightarrow${\tt (iii)}.\quad We show that if $h$ is continuous and non-constant then the inequality 
\begin{equation}\label{E:impossibleRh}
|\Rh (\mathbf z)|\leq C\quad \text{for any non-collinear three-tuple}\ \mathbf z
\end{equation}
is impossible. Fix a three-tuple $\mathbf z=\{0, z_2, z_3\}$ of non-collinear points with the property that the triangle $\Delta (0, z_2, z_3)$ has an obtuse angle at $z_3$, see Figure
 \ref{fig:Figure1}. Let $z_4$ denote the orthogonal projection of $z_3$ onto the opposite side of $\Delta (0, z_2, z_3)$, thus $z_4=\beta z_2$ with $0< \beta< 1$. Consider the family of triangles $\Delta \big(0, z_2, \zt (\theta)\big)$ where $\zt (\theta)$ lies along the line segment whose endpoints are  $z_3$ and $z_4$, see Figure 3.
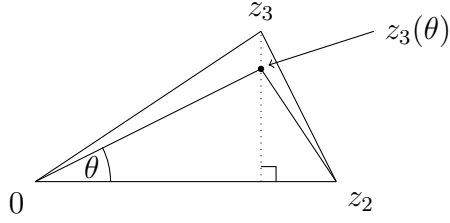
\begin{figure}[h!]\label{fig:Figure3}
     \centering

       \begin{tikzpicture}
   
    \draw (0,0)--(4,0);
    \draw (4,0) -- (3,2);
    \draw (0,0) -- (3,2);

    \draw (0,0) node[below left] {$0$};
    \draw (4,0) node[below right] {$z_2$};
    \draw (3,2) node[above] {$z_3$};

    \draw (0,0) -- (3,1.5) -- (4,0);
    \draw [dotted] (3,0) -- (3,2);
    \draw (3.2,0) -- (3.2,0.2) -- (3,0.2);
    \draw [fill=black] (3,1.5) circle (1pt);

    \draw [<-] (4.5,2) (3.1,1.55) -- (4.5,2) node [right] {$z_3(\theta)$};

    \draw (1,0) arc (0:27:1cm);
    \draw (0.75,0.2) node {$\theta$};
 \end{tikzpicture}
   \caption{Proof of  
   Theorem \ref{T:L-infty}:
   the triangles $\Delta(0, z_2, z_3)$ and $\Delta(0, z_2, z_3(\theta))$.}
   \end{figure}

Note that the three-tuples $(0, z_2, z_3)$  and $(0, z_2, z_3(\theta))$ are arranged in admissible order. Thus,
adopting the notation $\ell_1(\theta)$ and  $\ell_2(\theta)$  for the lengths of the sides of  $\Delta(0, z_2, z_3(\theta))$ opposite to $0$ and $z_2$, respectively, along with $\theta_2(\theta)$ for the angle at $z_2$, we may express $\Rh \big(0, z_2, \zt (\theta)\big)$ as the quotient 
\begin{equation}\label{E:CD}
\Rh \big(0, z_2, \zt (\theta)\big) = \frac{E \big(0, z_2, \zt (\theta)\big)}{F \big(0, z_2, \zt (\theta)\big)}
\end{equation}
with 
\begin{equation}\label{E:C}
E \big(0, z_2, \zt (\theta)\big)\ =\ 
\end{equation}
$$
=\ \lo (\theta )\cos\! \big(\mathfrak h_{0, z_2}(0) -\theta \big) +\, \ltw (\theta )  \cos\! \big(\mathfrak h_{0, z_2}(z_2) +\ttw (\theta )\big) -\, \lth\cos\! \big(\mathfrak h_{0, z_2}(z_3(\theta)) +\ttw (\theta ) -\theta\big)
$$
and
\begin{equation}\label{E:D}
F \big(0, z_2, \zt (\theta)\big) \ =\ 4\,\lth \sin\theta\, \sin\ttw (\theta)\, ,
\end{equation}
where the denominator $F\big(0, z_2, \zt (\theta)\big)$ comes from the identity
$$
\frac{\lo(\theta)\,\ltw(\theta)\,\lth}{(4\text{Area}\,\D \big(0, z_2, \zt (\theta)\big) )^2}\ =\
 \frac{1}{4\,\lth \sin\theta\, \sin\ttw (\theta)}\, 
$$
which in turn follows from \eqref{E:Marea-1} and
\eqref{E:Msin-1}.

 \noindent Let $\omega$ denote the direction of the vector $\vec{ 0\,z_2}$ that is,
 $\arg (\omega) =\ato$. Now letting $\theta\to 0$, we have that: $z_3 (\theta)\to z_4 =\beta z_2$; $\theta_2 (\theta) \to 0$; $\ltw (\theta)\to \beta\lth$, and 
  $\lo (\theta)\to (1-\beta)\lth$. By the continuity of $h(z)$ we also have that $\mathfrak h_{0, z_2}(z_3 (\theta))\to \mathfrak h_{0, z_2}(z_4)=\mathfrak h_{0, z_2}(\beta z_2)$. Inserting these into \eqref{E:C} and \eqref{E:D} we obtain that
$$
E \big(0, z_2, \zt (\theta)\big)\ \to\ \lth\big[(1-\beta)\cos\big(\mathfrak h_{0, z_2}(0)\big) + \beta\, \cos\big(\mathfrak h_{0, z_2}(z_2)\big)-\cos \big(\mathfrak h_{0, z_2}(\beta z_2)\big)\big],
$$
whereas
$$
F \big(0, z_2, \zt (\theta)\big) \to \lth\,\sin^2 0\,=\, 0.
$$
Thus, by the assumed boundedness of $\Rh$ (that is condition {\tt (i)}) we must have that
\begin{equation}\label{E:temp-2}
\cos (\mathfrak h_{0, z_2}(\beta z_2)) = (1-\beta)\cos (\mathfrak h_{0, z_2}(0)) +
 \beta \cos (\mathfrak h_{0, z_2}(z_2))
\end{equation}
for any $0<  \beta<  1$, for any direction $\omega$ and for any $z_2=|z_2|\omega$.
\vskip0.1in
\noindent Before proceeding any further, we recall the definition of $\mathfrak h_{0, z_2}(z)$:
$$
\mathfrak h_{0, z_2}(z) = 2h(z) -2\arg z_2\, ,
$$
thus in particular we have that
\begin{equation}\label{E:obs-1}
\mathfrak h_{0, \omega} (\zeta) = \mathfrak h_{0, t\omega}(\zeta)\quad \text{for any direction}\ \omega,
\ \ \text{ any\ }  t>0\ \  \text{and any}\ \zeta\in\mathbb C.
\end{equation}
\medskip
\noindent Given the direction $\omega$ as above, let us consider any point $z$ in the direction of  $\omega$, that is $z=|z|\omega$. We distinguish two cases:
\begin{itemize}
\item[\underline{{\em Case 1.}}] $|z| <  1$: Applying \eqref{E:temp-2} to $z_2:=\omega$ and 
$\beta =|z|$  we find that
\begin{equation}\label{E:temp-3}
\cos (\mathfrak h_{0, \omega}(z)) = (1-|z|)\cos (\mathfrak h_{0, \omega}(0)) + |z|\cos (\mathfrak h_{0, \omega}(\omega))
\end{equation}

\qquad\ \  for any direction $\omega$ and for any $z=|z|\omega$ with $|z|\leq1$.
\item[]
\item[\underline{{\em Case 2.}}]\ $|z|>1$: In this case we apply \eqref{E:temp-2} to $z_2:=z = |z|\omega$ and $\beta =1/|z|$. Taking 

\qquad\quad \eqref{E:obs-1} into account it is easy to see that \eqref{E:temp-3}  also holds for $|z|>1$.
\item[]
\end{itemize}
By continuity it follows that \eqref{E:temp-3} must hold also for $|z|=1$.   We conclude that the identity
\begin{equation}\label{E:temp-4}
\cos (\mathfrak h_{0, \omega}(z)) = \cos (\mathfrak h_{0, \omega}(0)) + |z|\big(\cos(\mathfrak h_{0, \omega}(\omega))-\cos (\mathfrak h_{0, \omega}(0))\big)
\end{equation}
 holds for any direction $\omega$ and for any $z=|z|\omega$. However note that the
  left-hand side of \eqref{E:temp-4} is $O(1)$, while the right-hand side is 
  $C_\omega + O(|z|)$, thus the only possibility is that 
  $$
  \cos (\mathfrak h_{0, \omega}(\omega)) = \cos (\mathfrak h_{0, \omega}(0)) \equiv C_\omega \quad \text{for \ any}\quad |\omega|=1\, .
  $$
  Substituting the latter into \eqref{E:temp-4} we obtain
 
  \begin{equation}\label{E:temp-5}
  \cos (\mathfrak h_{0, \omega}(z)) = C_\omega \quad \text{for\ any}\ |\omega|=1\quad \text{and \ for\ any}\ z=|z|\omega.
  \end{equation}
  But recall that $\mathfrak h_{0, \omega}(z) = 2h(z) -2\arg (\omega)$; thus it follows from
  \eqref{E:temp-5} that
   \begin{equation}\label{E:temp-6}
 h(z) = \arg (\omega) +  \widetilde C_\omega \quad \text{for\ any}\ |\omega|=1\quad \text{and \ for\ any}\ z=|z|\omega\, , 
  \end{equation}
  which implies that
  $$
  h(z) = f\left(\frac{z}{|z|}\right),\quad z\in\mathbb C\setminus\{0\}.
  $$
From this we conclude that $f$, and thus $h$, must be constant: if not,  there would be two directions $\omega_1\neq\omega_2$ such that   $f(\omega_1)\neq f(\omega_2)$. By the assumed continuity of $h(z)$ it would then follow that
$$
h(0) = \lim\limits_{r\to 0} h(r\omega_1) = f(\omega_1)\neq f(\omega_2) =
\lim\limits_{r\to 0} h(r\omega_2) = h(0) ,
$$
which is a contradiction.
 The proof of {\tt (i.)} $\Rightarrow$ {\tt (iii)} is concluded, and so is the proof of \cite[Theorem 5.3]{LPr}.
 \vskip0.1in
\noindent {\bf Remark (C).} It is possible that the hypothesis that $h:\mathbb C\to\mathbb R$ is continuous may be relaxed.
Below we give an example where $h(z)$ is constant except at one point and we show that the corresponding $\Rh$ is unbounded.
\vskip0.1in
\noindent Fix $\epsilon_0>0$ such that $\sin \epsilon_0>1/2$ and set
\begin{equation}
\displaystyle{h(z) = \left\{
\begin{array}{lll}
0& \text{if}\ z\neq 0,\\
\frac{\epsilon_0}{2} & \text{if}\ z=0
\end{array}
\right.
 } 
\end{equation}
Consider three-tuples of the form $\mathbf z_\lambda=\{0, 1, i\,\lambda\}$ with $\lambda>0$. Then
 it is easy to see that
$$
\Rh (\mathbf z_\lambda)\ =\ \frac{1+\lambda^2}{\lambda}\, \sin\epsilon_0\, .
$$
Now the condition 
$
|\Rh (\mathbf z)|\,\leq C
$
for all three-tuples of non-collinear points would, in particular, require that
$$
\frac12<\,\sin\epsilon_0\leq\  \frac{\lambda}{1+\lambda^2}\,C\quad \text{for any}\ \lambda>0\, ,
$$
which is not possible.
\vskip0.1in

\subsection{Proof of
Corollary \ref{C:L-infty}}
 If $h$ is constant then by 
 Theorem \ref{T:L-infty} 
 we have that $\Rh (\mathbf z) = 0$ for all three-tuples of non-collinear distinct points, thus 
$$
\left(\frac12 -\Rh (\mathbf z)\right)\!\! \left(\frac12 +\Rh (\mathbf z)\right)= \frac14>0.
$$
Conversely, if 
$$
\left(\frac12 -\Rh (\mathbf z)\right)\!\! \left(\frac12 +\Rh (\mathbf z)\right)> 0
$$
for all three-tuples of non-collinear points then
$$
|\Rh (\mathbf z)|<\frac12
$$
and 
Theorem \ref{T:L-infty}
 gives that $h$ is constant.
 \vskip0.1in
\subsection{Proof of
Theorem \ref{T:positive}} 

Conclusion {\tt (a)} is immediate
from 
Theorem \ref{T:L-infty},
 thus we only need to prove conclusion {\tt (b)}. Recall the definition 
$\mathfrak h_{z_1, z_2}(z) := 2h(z) - 2\text{Arg}(z_2-z_1)$ (for an arbitrarily fixed coordinate system).

\begin{lem}\label{L:Properties-3} Suppose that $h:\mathbb C\to \mathbb R$ is continuous and non-constant. Then there exist $z_1\neq z_2$ and  $t_0\in (0, 1)$ such that for $z_0=t_0z_2+(1-t_0)z_1$ we have
\begin{equation}\label{E:P-4}
\cos\mathfrak h_{z_1, z_2}(z_1) - \cos\mathfrak h_{z_1, z_2}(z_0) \, <\, 0\, < \cos\mathfrak h_{z_1, z_2}(z_2) - \cos\mathfrak h_{z_1, z_2}(z_0)\, .
\end{equation}
\end{lem}
\begin{proof}
It suffices to find $z_1\neq z_2$ such that 
\begin{equation}\label{E:less}
\cos\mathfrak h_{z_1, z_2}(z_1) < \cos\mathfrak h_{z_1, z_2}(z_2).
\end{equation}
The existence of $z_0$ will then then follow by the mean-value theorem applied to 
$f(t):= \cos\mathfrak h_{z_1, z_2}\big(tz_2 +(1-t)z_1\big)$).  We proceed by contradiction and suppose that 
 \begin{align*}  \cos\mathfrak h_{z_1, z_2}(z_1) &= \cos\mathfrak h_{z_1, z_2}(z_2) \text{ for all } z_1\neq z_2,  
\text{ that is, } \\  \cos\left (2h(z_1) - 2\text{Arg}(z_2-z_1)\right) &= \cos\left (2h(z_2) - 2\text{Arg}(z_2-z_1)\right) \text{ for all }  z_1\neq z_2.
\end{align*} 
But $\cos A = \cos B$ if and only if either $B - A \in 2 \pi \mathbb Z$ or $B+A \in 2 \pi \mathbb Z$. Thus for any $z_1 \ne z_2$, there exists $k = k(z_1, z_2) \in \mathbb Z$ such that one of the two possibilities holds:
\[ 2h(z_2) - 2\text{Arg}(z_2-z_1) = \begin{cases} 2h(z_1) - 2\text{Arg}(z_2-z_1) + 2k \pi, \text{ or } \\ 2k \pi - 2h(z_1) + 2\text{Arg}(z_2-z_1). \end{cases} \]
Equivalently stated, 
\begin{equation} \label{h-discrete}  h(z_2) \in \bigl\{h(z_1) + k \pi : k \in \mathbb Z \bigr\} \cup \bigl\{k \pi - h(z_1) + 2 \text{Arg}(z_2-z_1) : k \in \mathbb Z \bigr\}.  \end{equation} 
If $h$ is continuous and non-constant, there exist $z_0, \omega \in \mathbb C$ and $|\omega| = 1$ such that the map
\[ t \in \mathbb R \mapsto h(z_0 + t \omega) \text{ is continuous and non-constant.} \]
Hence, the image set $h(\mathscr{L})$  contains an interval, where $\mathscr{L} = \{z_0 + t \omega : t \in \mathbb R\}$.  However setting $z_1=z_0$ and $z_2 \in \mathscr{L}$ in \eqref{h-discrete}, we find that
\[ h (\mathscr{L}) \subseteq \bigl\{ h(z_0) + k \pi : k \in \mathbb Z \bigr\} \cup \bigl\{ k \pi - h(z_0) + 2\omega : k \in \mathbb Z \bigr\}. \] 
The right hand side above is a discrete set, whereas the left contains an interval, providing the desired contradiction. 
\end{proof}

\noindent {\em Proof of {\tt (b)}.}
Let $z_1, z_0$ and $z_2$ be as in Lemma \ref{L:Properties-3}. 
 Consider two families of non-collinear three-tuples $\displaystyle{\{\mathbf z_R(\theta)\}_{\theta, R}}$ and
 $\displaystyle{\{\mathbf z_{R'}(\theta')\}_{\theta', R'}}$ defined as follows:

\begin{equation}\label{E:family-1}
\mathbf z_R(\theta) :=  \big(z_0, z_2^R, z_3(\theta)\big),\quad \theta\in \left(0, \frac{\pi}{2}\right), \ \ R>0,
\end{equation}
where 
$$
z_2^R := (z_2-z_0)(1+R) + z_0
$$
and the point $z_3(\theta)$ has been chosen  so that $\displaystyle{\lim\limits_{\theta\to 0}z_3(\theta)=z_2}$, see Figure 4; and 
\begin{equation}\label{E:family-2}
\mathbf z_R(\theta') :=  \big(z_1^{R'}, z_0, z_2^R, z_3(\theta')\big),\quad \theta'\in \left(0, \frac{\pi}{2}\right), \ \ R'>0,
\end{equation}
where 
$$
z_1^{R'} := R'(z_1-z_0) + z_1
$$
and the point $z_3(\theta')$ has been chosen  so that $\displaystyle{\lim\limits_{\theta'\to 0}z_3(\theta')=z_1}$, see again Figure 4.
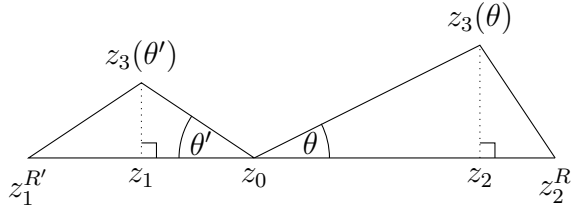
\begin{figure}[h!]\label{Figure 4}
     \centering

       \begin{tikzpicture}
   
    \draw (0,0)--(4,0);

    \draw (0,0) node[below ] {$z_0$};
    \draw (3,0) node[below] {$z_2$};

    \draw (0,0) -- (3,1.5) -- (4,0);
    \draw [dotted] (3,0) -- (3,1.5);
    \draw (3.2,0) -- (3.2,0.2) -- (3,0.2);

    \draw (3,1.5) node [above] {$z_3(\theta)$};

    \draw (1,0) arc (0:27:1cm);
    \draw (0.75,0.2) node {$\theta$};

    \draw (0,0) -- (-3,0) -- (-1.5,1) -- (0,0);
    \draw (-1.5,1) node [above] {$z_3(\theta')$};
    \draw (-1,0) arc (180:146:1cm);
    \draw (-1,-0.1) node [above right] {$\theta'$};
    \draw [dotted] (-1.5,1) -- (-1.5,0) node [below] {$z_1$};
    \draw (-1.3,0) -- (-1.3,0.2) -- (-1.5,0.2);
    \draw (-3,0) node [below] {$z_1^{R'}$};
    \draw (4,0) node [below] {$z_2^R$};

 \end{tikzpicture}
 \caption{Proof of
 Theorem \ref{T:positive}: 
 the families $\displaystyle{\{\mathbf z_R(\theta)\}_{\theta, R}}$ and
$\displaystyle{\{\mathbf z_{R'}(\theta')\}_{\theta', R'}}$.}
  
   \end{figure}

Note that the three-tuples $(z_0, , z_2^{R}, z_3(\theta))$ and $(z_1^{R'}, z_0, z_3(\theta'))$ are arranged in admissible order. We claim that there exist small $\theta_0$ and $\theta'_0$, and large $R$ and $R'$ such that
\begin{equation}\label{E:claim-triple-1}
\frac12 +\Rh (\mathbf z_{R_0}(\theta_0))<0
\end{equation}
and 
\begin{equation}\label{E:claim-triple-2}
\frac12 +\Rh (\mathbf z_{R'_0}(\theta_0'))>0\, .
\end{equation}
To prove claim \eqref{E:claim-triple-1} define
$$\mathfrak{h}(z) := \mathfrak{h}_{z_1, z_2}(z),\quad z\in\mathbb C.$$
\vskip0.1in
\noindent We point out that since $\mathfrak{h}_{z_1, z_2}(z)$ depends on $z_1$ and $z_2$ only through
$\text{Arg}(z_2-z_1)$ it follows that

$$
\mathfrak{h}_{z_1, z_2}(z) = \mathfrak{h}_{\tilde{z}_1, \tilde{z}_2}(z)
$$
whenever $$\text{Arg}(z_2-z_1) = \text{Arg}(\tilde{z}_2-\tilde{z}_1).$$In particular we have that
$$
\mathfrak{h}(z)=\mathfrak{h}_{z_0, z_2^R}(z)= \mathfrak{h}_{z_1^R, z_0}(z)\, .
$$
Invoking 
\eqref{E:RH}
 to compute $1/2 + \Rh (\mathbf z_R(\theta))$ 
and using notation analogous to \eqref{E:CD} we see that 
$$
\text{sign}\left(\frac12 + \Rh (\mathbf z_R(\theta))\right) = \text{sign}\bigg(F(\mathbf z_R(\theta)) + 2\,E(\mathbf z_R(\theta))\bigg) \quad \text{for all}\ \ \theta>0,\ \ R>0,
$$
where
\begin{align*}
F(\mathbf z_R(\theta)) &= 4 (1+ R) |z_2-z_0|\sin\theta\,\sin\theta_2(\theta, R), \text{ and } \\ 
 E(\mathbf z_R(\theta))  &=\ell_1^R(\theta)\cos\left(\mathfrak{h}(z_0) -\theta\right)+\,
 \ell_2(\theta)\cos\left(\mathfrak{h}(z_2^R)+\theta_2(\theta, R)\right)\\ &\qquad - |z_2-z_0|(1+R)\cos\left(\mathfrak{h}(z_3(\theta) + \theta_2(\theta, R) -\theta\right).
 \end{align*}
 We see that, for fixed $R>0$, 
 \begin{align*}
 \lim\limits_{\theta\to 0} \Bigl(F(\mathbf z_R(\theta)) &+ 2E(\mathbf z_R(\theta))\Bigr)\\ 
 &= 2|z_2-z_0|\Bigl(R\big(\cos\mathfrak{h}(z_0) -\cos\mathfrak{h}(z_2)\big)  +
  \cos\mathfrak{h}(z_2^R) -\cos\mathfrak{h}(z_2)\Bigr).
 \end{align*} 
 But \eqref{E:P-4} tells us that the latter is negative for sufficiently large $R$. Thus, by continuity, there are
 $\theta_0 \ll 1$ and $R_0 \gg 1$ such that \eqref{E:claim-triple-1} holds.
  \vskip0.1in
  The proof of claim \eqref{E:claim-triple-2} is similar: proceeding as above we find that
$$
\text{sign}\Bigl(\frac12 + \Rh (\mathbf z_R(\theta'))\Bigr) = \text{sign}\Bigl(F(\mathbf z_R(\theta')) + 2\,E(\mathbf z_R(\theta'))\Bigr) \quad \text{for all}\ \ \theta>0,\,R>0,
$$
where
\begin{align*}
F(\mathbf z_R(\theta')) &= 4(1+R)|z_1-z_0|\sin\theta'\,\sin\theta_1\!(\theta', R), \text{ and } \\ 
 E(\mathbf z_R(\theta')) &= \ell_1(\theta')\cos\left(\mathfrak{h}(z_1^R) -\theta_1(\theta', R)\right)
 +\, \ell_2^R(\theta')\cos\left(\mathfrak{h}(z_0)+\theta'\right)\\ & \qquad \qquad - |z_1-z_0|(1+R)\cos\left(\mathfrak{h}(z_3(\theta') + \theta' -\theta_1(\theta', R)\right).
\end{align*}
We obtain (again for fixed $R>0$)
 \begin{align*} 
 \lim\limits_{\theta'\to 0} \Bigl(F(\mathbf z_R(\theta')) &+ 2E(\mathbf z_R(\theta))\Bigr) \\ 
 &= 2|z_1-z_0|\Bigl(R\big(\cos\mathfrak{h}(z_0) -\cos\mathfrak{h}(z_1)\big)  +
  \cos\mathfrak{h}(z_1^R) -\cos\mathfrak{h}(z_1)\Bigr),
 \end{align*} 
which is positive by \eqref{E:P-4}, for $R$ large enough. Therefore, by continuity, there are
$\theta'_0 \ll 1$ and $R_0' \gg 1$ such that
 \begin{equation}\label{E:pos}
 \frac12 + \Rh\left(\mathbf z_{R'_0}(\theta'_0)\right) >0.
 \end{equation}
The proof of
Theorem \ref{T:positive} 
is concluded.


\subsection{Proof of 
Proposition \ref{P:symm-id-K-star}}
On account of \eqref{E:ReK-simpl} and \eqref{E:Symm-simpl2} we have
\begin{equation}\label{E:AdjSym1}
{\tt S}[\,K_h^\ast\,](\mathbf z)
 =\, \frac{2}{\ell_1^2\ell_2^2\ell_3^2}
\sum\limits
_{\stackrel{j}{k<l}} \ell^2_j\,
\mbox{Re}
\bigg(e^{-i\big(h(z_k)-h(z_l)\big)}(z_k-z_j)(\overline{z_l-z_j})\bigg)\, .
\end{equation}
\vskip0.05in
Adopting the labeling scheme for $\D(\mathbf z)$ that was described in Section \ref{SS:LabelingScheme},
we expand the above sum and invoke \eqref{E:trig-id-1} and \eqref{E:trig-id-2} to express each of $(z_k-z_j)$, resp. $(\overline{z_l-z_j})$, in terms of $(z_2-z_1)$, resp. $(\overline{z_2-z_1})$. This leads us to the identity
$$
{\tt S}[\,K_h^\ast\,](\mathbf z) = \
\frac{2}{\ell_1^2 \ell_2^2 \ell_3^2}\Big[\ell_1^2\ell_3^2\,\frac{\beta}{\cos\theta_1}\cos\big(h(z_2)-h(z_3)+
\theta_1\big)+
$$
$$
\ell_2^2\ell_3^2\,\frac{(1-\beta)}{\cos\theta_2}\cos\big(h(z_1)-h(z_3)-
\theta_2\big) -
\ell_3^4\,\frac{\beta}{\cos\theta_1}\frac{(1-\beta)}{\cos\theta_2}\cos\big(h(z_1)-h(z_2)-
(\theta_1+\theta_2)\big)
\Big].
$$
The latter simplifies to
$$
{\tt S}[\,K_h^\ast\,](\mathbf z) = \
\frac{2}{\ell_1 \ell_2 \ell_3}\Big[
\ell_1\ell_3\frac{\beta}{\ell_2\cos \theta_1}\cos\big(h(z_2)- h(z_3) +\theta_1\big) 
+ 
$$
$$
+\ell_2\ell_3\frac{(1-\beta)}{\ell_1\cos \theta_2}\cos\big(h(z_1)- h(z_3) -\theta_2\big) 
-
\ell_3^3\frac{\beta}{\ell_2\cos\theta_1}\frac{(1-\beta)}{\ell_1\cos \theta_2}\cos\big(h(z_1)- h(z_2) -(\theta_1+\theta_2)\big)
\Big].
$$
Recalling
\eqref{E:aux-one} and \eqref{E:elem-trig}  we conclude that
\begin{equation*}\label{E:repr-adj}
{\tt S}[\,K_h^\ast\,](\mathbf z) = \ c^2(\mathbf z)\,\H (\mathbf z)
\end{equation*}
with $\H (\mathbf z)$ as in \eqref{E:H}.
The proof is concluded.
\vskip0.1in
We point out that the arguments in Remark {\bf (B)} (showing the well posedness of the quantity $c^2(z)\Rh (\mathbf z)$ for triples of distinct (but possibly collinear) points) apply verbatim to $c^2(\mathbf z)\,\H (\mathbf z)$.


\subsection{Proof of 
Theorem \ref{T:L-infty-H}}
\label{SS:H-unbd}
\noindent Since the implication {\tt (ii)} $\Rightarrow$ {\tt (i)} is trivial, we only need to prove
that
 \medskip
 
 \centerline{{\tt (iii)} $\Rightarrow$ {\tt (ii)}\quad  and\quad   {\tt (i)} $\Rightarrow$ {\tt (iii)}.}
\vskip0.07in

\noindent We first show that {\tt (iii)} $\Rightarrow$ {\tt (ii)}. If $h(z)=$ const. then Proposition \ref{P:symm-id-K-star} 
 gives that
$$
\H (\mathbf z)\,=\, 2\,\frac{\ell_1\ell_2\ell_3}{\left(4\text{Area}\Delta(\mathbf z)\right)^2}
\big\{\ell_1\cos\theta_1 +\ell_2\cos\theta_2+\ell_3\cos\theta_3\big\}
$$
and it follows from \eqref{E:aux-one} and \eqref{E:Mcos-2} that the latter equals 1.
\vskip0.1in
\noindent We are left to prove the implication {\tt (i)} $\Rightarrow$ {\tt (iii)}. Equivalently, we show that if $h$ is continuous and non-constant then the inequality 
\begin{equation}\label{E:impossibleH}
|\H (\mathbf z)|\leq C\quad \text{for any non-collinear three-tuple}\ \mathbf z
\end{equation}
is impossible. To this end, first note that ${\tt S}[\,K_h^\ast\,] = {\tt S}[\,K_{h+c}^\ast\,]$, for any constant $c$. Thus we may assume without loss of generality that $h(0) = 0$. Define
\begin{align} S &:= \{z \in \mathbb C : h(z) \not\in 2 \pi \mathbb Z \},  \text{ and } \\  \mathcal A_z &:= \Bigl\{ \cot \Bigl(\frac{h(z)}{4} \Bigr), - \tan \Bigl( \frac{h(z)}{4} \Bigr)  \Bigr\}, \quad \text{ for } z \in S. 
\end{align}  
Note that $h(z)\neq 0$ for each $z\in S$, and in particular $0\notin S$.
\begin{lem}\label{L:Properties-1}
 $S$ is nonempty, and hence $\mathcal A_z$ is well-defined for every $z \in S$.
  \end{lem}
\begin{proof} 
If $S = \emptyset$, this means either $h \equiv 0$, or that $h(\mathbb C)$ is disconnected. The first case contradicts the fact that $h$ is nonconstant, and the second contradicts its continuity.    
\end{proof} 
\begin{lem}\label{L:Properties-2}
For every $z_0 \in S$ there exist two numbers $s \neq t \in \mathbb R^+$ such that
\begin{align} \label{E:1}
& \{sz_0, tz_0\} \subseteq S, \text{ and } \\   &\mathcal A_{sz_0} \cap \mathcal A_{tz_0}= \emptyset. \label{E:2}
\end{align}   
\end{lem} 
\begin{proof} 
By continuity of $h$, the function $g: \mathbb R \rightarrow \mathbb R$ given by \[g(t) := h(t z_0) \] is a non-constant continuous function on $\mathbb R$, and we have that $g(0)=0$, $g(1) \neq 0$.  The intermediate value theorem ensures that for any small $\epsilon > 0$, there exists $t_0 \in (0,1)$ such that $ 0 < |g(t_0)| < \epsilon < 2\pi$. Let us choose $\epsilon > 0$ small enough so that \begin{align*} 
&\alpha \mapsto \cot(\alpha/4), \quad \alpha \mapsto - \tan(\alpha/4) \text{ are both injective on } |\alpha| < \epsilon, \text{ and} \\  &\{\cot(\alpha/4) : |\alpha| < \epsilon  \} \cap \{ - \tan(\alpha/4) : |\alpha| < \epsilon \} = \emptyset. \end{align*} Without loss of generality, assume $g(t_0) > 0$. By the intermediate value theorem again, $g$ assumes every value between $g(0) = 0$ and $g(t_0)$ on the interval $[0, t_0]$. In particular, let $s$ and $t$ denote two distinct points in $(0, t_0)$ such that $g(s) \ne g(t)$ and $g(s), g(t) \in (0, g(t_0))$.  The conclusions of the lemma then hold for this choice of $s$ and $t$. 
 \end{proof}

 \vskip0.2in

\noindent We now continue with the proof of the implication {\tt (i)} $\Rightarrow$ {\tt (iii)} in
 Theorem \ref{T:L-infty-H}. 
Let $0\neq z_0\in S$ and 
 $s\neq t\in \mathbb R^+$  be as in Lemmas \ref{L:Properties-1} and \ref{L:Properties-2}, respectively. 
 Set 
 
 $$\zt:= sz_0\in S\quad \text{and}\quad  
 \zt':= tz_0\in S\, .$$
 Thus, the points $0, \zt$ and $\zt'$ are distinct but collinear, and they determine a ray $\mathcal L_0$ depicted in Figure 
 5 below. Now let $\beta>1$ be given and define two points
 $$
 z_4(\beta) := \beta z_3\in \mathcal L_0,\quad \text{and} \quad z_4'(\beta) := \beta \zt' \in \mathcal L_0
 $$
 
\noindent  For any $\theta\in (0, \pi/2)$ let $\mathcal L_\theta$ be (any fixed) ray forming an angle $\theta$ with $\mathcal L_0$, see Figure 5.
 
 \vskip0.05in

Next we let 
$$z_1(\theta, \beta)\in \mathcal L_\theta\quad \text{and}\quad z_1'(\theta, \beta)\in \mathcal L_\theta
$$ 
be the intersection points with $\mathcal L_\theta$ of the lines perpendicular to $\mathcal L_0$ and passing through $z_4 (\beta)$ and $z_4'(\beta)$, respectively. It is easy to check that the three-tuples 
$(z_1(\theta, \beta), 0, \zt)$ and $(z_1'(\theta, \beta), 0, \zt')$ are arranged in admissible form: see Figure 5.
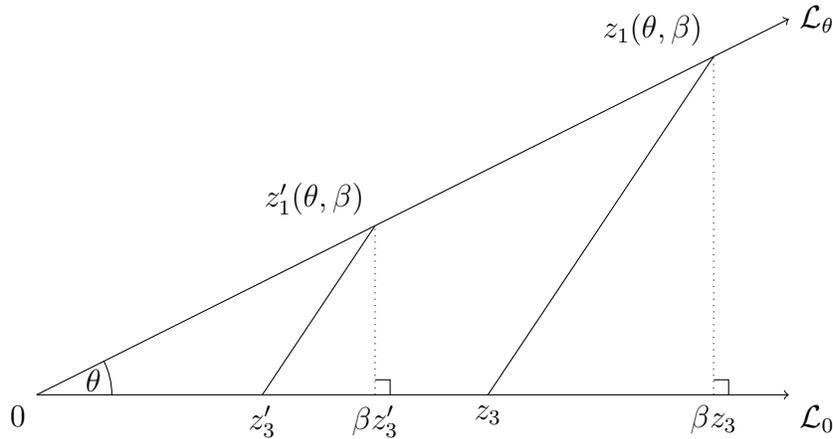
\begin{figure}[h!]\label{Figure 5}
     \centering

       \begin{tikzpicture}
   
    \draw [->] (0,0)--(10,0) node [below right] {$\mathcal{L}_0$};
    \draw [->] (0,0) -- (10,5) node [right]  {$\mathcal{L}_\theta$};
    \draw (3,0)  node [below] {$z_3'$} -- (4.5,4.5/2) node [above left] {$z_1'(\theta,\beta)$};
    \draw (6,0) node [below] {$z_3$} -- (9,4.5) node [above left] {$z_1(\theta,\beta)$};
    \draw [dotted] (4.5,0) node [below] {$\beta z_3'$} -- (4.5,4.5/2);
    \draw [dotted] (9,0) node [below] {$\beta z_3$} -- (9,4.5);
    \draw (4.7,0) -- (4.7,0.2) -- (4.5,0.2);
    \draw (9.2,0) -- (9.2,0.2) -- (9.0,0.2);
    \draw (0,0) node [below left] {$0$};

    \draw (1,0) arc (0:27:1);
    \draw (1,-0.1) node [above left] {$\theta$};
   
     \end{tikzpicture}
\caption{Proof of  
Theorem \ref{T:L-infty-H}: 
the three-tuples $(z_1'(\theta, \beta), 0, \zt')$ and $(z_1(\theta, \beta), 0, \zt)$.}   
   \end{figure}

\bigskip

\noindent Let us focus for a moment on the family of triangles determined by the first three-tuple, $\Delta(z_1(\theta, \beta), 0, \zt)$: 
it follows from 
\eqref{E:H}
 that 
\[ \frac{(4\text{Area}(\Delta))^2\,\H (z_1(\theta, \beta), 0, z_3)}{\lo\,\ltw(\theta, \beta)\,\lth (\theta, \beta)}\,=\, A(\theta, \beta, \zt)\, \]
holds for any $\beta>1$ and $0<\theta<\pi/2$, where we have set
\begin{align*}
A(\theta, \beta, \zt) &:= 
\lo\cos(h(\zt) + \theta)\, +\, \ltw(\theta, \beta) \cos (h(z_1(\theta, \beta)-h(\zt) -\theta) \\ &\qquad  \qquad -
\lth(\theta, \beta)\cos(h(z_1(\theta, \beta))-\theta_1(\theta, \beta)-\theta).
\end{align*} 
(Recall that $h(0)=0$.)
Next we point out that
\[ 4\text{Area}(\Delta) = 2\lo^2\,\beta\tan\theta\, ;\; \ltw (\theta, \beta) =
 \lo\left(\frac{\beta^2}{\cos^2\theta}+ (1-2\beta)\right)^{\frac12}, \text{and}\; 
 \lth(\theta, \beta)=\lo\,\frac{\beta}{\cos\theta}. \] 
Thus if \eqref{E:impossibleH} were to hold, we would have that 
$$
\lim\limits_{\theta\to 0}\, \frac{(4\text{Area}(\Delta))^2\, \H (z_1(\theta, \beta), 0, z_3)}{\lo\,\ltw(\theta, \beta)\,\lth (\theta, \beta)}\, = 0 \quad \text{for all}\ \ \beta>1,
$$
and this, in turn, would imply that
$$
0= \lim\limits_{\theta\to 0}\, A(\theta, \beta, \zt) = 
\lo\bigg(\!\!\cos \big(h(\zt)\big) + (\beta-1)\cos \big(h(\beta \zt) - h(\zt)\big)-\beta\cos\big(h(\beta\zt)\big)\!\bigg)
$$
for all $\beta>1$, that is
\begin{equation}\label{E:A=zero}
\cos \big(h(\zt)\big) + (\beta-1)\cos \big(h(\beta \zt) - h(\zt)\big)-\beta\cos\big(h(\beta\zt)\big)\ =\ 0
\quad \text{for all}\quad \beta>1.
\end{equation}
\vskip0.1in
\noindent Applying this same reasoning to the family of triangles 
$\Delta (z_1'(\theta, \beta), 0, \zt')$,
we similarly obtain that
\begin{equation}\label{E:A-prime=zero}
\cos \big(h(\zt')\big) + (\beta-1)\cos \big(h(\beta \zt') - h(\zt')\big)-\beta\cos\big(h(\beta\zt')\big)\ =\ 0
\quad \text{for all}\quad \beta>1.
\end{equation}
\vskip0.1in
\noindent In summary: if \eqref{E:impossibleH} were to hold for all non-collinear three-point configurations, then \eqref{E:A=zero} and \eqref{E:A-prime=zero} would have to be true for all $\beta>1$.
\vskip0.1in
\noindent Applying the half-angle identities: 
$$
\sin\! \big(\alpha\big)\! = 2\sin\!\left(\frac{\alpha}{2}\right)\cos\!\left(\frac{\alpha}{2}\right);\qquad 
\cos\! \big(\alpha\big) = \cos^2\!\left(\frac{\alpha}{2}\right) -\,
\sin^2\!\left(\frac{\alpha}{2}\right)\ \ 
$$
$$
\text{to}\ \ \alpha:= h(\beta\zt)
$$
we see that \eqref{E:A=zero} has the equivalent formulation
\begin{equation}\label{E:A=zero-half}
(1-\beta)\left\{\cos(h(\zt))\bigg[\cos^2\!\left(\frac{\alpha}{2}\right)-\sin^2\!\left(\frac{\alpha}{2}\right)\bigg]
+ 2\sin(h(\zt))\sin\!\left(\frac{\alpha}{2}\right)\cos\!\left(\frac{\alpha}{2}\right)\right\}\, +
\end{equation}
$$
\, +\,
\beta\bigg[\cos^2\!\left(\frac{\alpha}{2}\right)-\sin^2\!\left(\frac{\alpha}{2}\right)\bigg] 
 -\cos(h(\zt))[\cos^2\!\left(\frac{\alpha}{2}\right)+\sin^2\!\left(\frac{\alpha}{2}\right)\bigg] =0\quad
 \text{for all}\ \beta>1.
$$
\vskip0.1in

\noindent A corresponding formulation holds for \eqref{E:A-prime=zero} (with $\alpha:= h(\beta\zt')$). There are now two possibilities:
\vskip0.2in
\begin{itemize}
\item[\underline{{\em Case 1:}}]\quad $\beta\zt \in S$ and $\beta\zt'\in S$ for all $\beta>1$.
\vskip0.2in
\item[\underline{{\em Case 2:}}]\quad $\beta_0\zt\notin S$ and/or $\beta_0'\zt'\notin S$ for some
$\beta_0$ and/or $\beta_0'>1$.
\end{itemize}
We show that each of {\em Case 1} and {\em Case 2} is either impossible, or leads to a contradiction.
\vskip0.2in
\underline{{\em Analysis of Case 1:}}\quad In this case we have that 
$$
\sin\!\left(\frac{h(\beta\zt)}{2}\right)\neq 0\quad \text{and}\quad \sin\!\left(\frac{h(\beta\zt')}{2}\right)\neq 0\quad \text{for all}\ \beta>1.
$$
Thus \eqref{E:A=zero} 
can be equivalently formulated
as
\begin{equation}\label{E:A=zero-case1}
\beta\big(1-\cos(h(\zt)\big)Y_\beta^2 + 2(1-\beta)\sin (h(\zt))Y_\beta -\bigg(\beta +(2-\beta)\cos(h(\zt)\bigg) = 0
\end{equation}
 for all $\beta > 1$, where we have set
\begin{equation}\label{E:Y}
Y_\beta:= \cot\left(\frac{h(\beta\zt)}{2}\right)\, .
\end{equation}
Similarly, we have that \eqref{E:A-prime=zero} is restated as
\begin{equation}\label{E:A-prime=zero-case1}
\beta\big(1-\cos(h(\zt')\big)Z_\beta^2 + 2(1-\beta)\sin (h(\zt'))Z_\beta -\bigg(\beta +(2-\beta)\cos(h(\zt')\bigg) = 0
\end{equation}
 for all $\beta > 1$, with
\begin{equation}\label{E:Z}
Z_\beta:= \cot\left(\frac{h(\beta\zt')}{2}\right).
\end{equation}
Note that \eqref{E:A=zero-case1} 
is a quadratic equation with discriminant $D(\beta)$ that satisfies
$$
\frac{D(\beta)}{4} = 2\big(1-\cos(h(\zt)\big)\beta^2 - 2\big(1-\cos(h(\zt)\big)\beta + \sin^2\big(h(\zt)\big)\, ,\quad \text{for all}\ \ \beta>1.
$$
Thus $D(\beta)$ lies on an up-ward looking parabola with vertex at $\beta_0=1/2$ and it follows that
$$
D(\beta)\, >\, 4\,\sin^2\big(h(\zt)\big)\geq 0\quad \text{for all}\ \ \beta>1.
$$
Similarly, the discriminant $D'(\beta)$ of the quadratic equation \eqref{E:A-prime=zero-case1} has
$$
D'(\beta)\, >\, 4\,\sin^2\big(h(\zt')\big)\geq 0\quad \text{for all}\ \ \beta>1.
$$
It follows that each of 
\eqref{E:A=zero-case1}
 and \eqref{E:A-prime=zero-case1} can be solved algebraically, giving us that
$$
\cot\!\left(\!\frac{h(\beta\zt)}{2}\!\right)
$$
equals one of the two quantities
$$
\frac{(\beta -1)\sin(h(\zt))\pm\big\{(\beta -1)^2\sin^2(h(\zt))+ \beta\big(1-\cos(h(\zt)\big)\big(\beta+(2-\beta)\cos(h(\zt))\big)\big\}
}
{\beta\big(1-\cos(h(\zt)\big)}^{\!\!1/2}
$$
for any $\beta>1$. We now let $\beta\to\infty$. One can verify that 
quantities above  converge as $\beta\to\infty$ with limits
\begin{equation*}\label{E:limitY}
\cot\left(\frac{h(\zt}{4}\right);\ -\tan\left(\frac{h(\zt}{4}\right)
\quad \text{respectively for the ``+'' and ``-'' determinations}.
\end{equation*}

We similarly have that 
$
\displaystyle{\cot\!\left(\!\frac{h(\beta\zt')}{2}\right)}
$
equals one of
$$
\frac{(\beta -1)\sin(h(\zt'))\pm\big\{(\beta -1)^2\sin^2(h(\zt'))+ \beta\big(1-\cos(h(\zt')\big)\big(\beta+(2-\beta)\cos(h(\zt'))\big)\big\}
}
{\beta\big(1-\cos(h(\zt')
\big)}^{\!\!1/2}\, ,
$$
and the latter converge as $\beta\to\infty$ to
$$
\cot\left(\frac{h(\zt'}{4}\right);\ -\tan\left(\frac{h(\zt'}{4}\right)\quad \text{respectively}.
$$
It follows that 
$$
\lim\limits_{\beta\to\infty} \cot\!\left(\!\frac{h(\beta\zt)}{2}\!\right)\ \in\  \left\{\cot\left(\frac{h(\zt}{4}\right);\, -\tan\left(\frac{h(\zt}{4}\right)\right\}
$$
and
$$
\lim\limits_{\beta\to\infty} \cot\!\left(\!\frac{h(\beta\zt')}{2}\!\right)\ \in\  \left\{\cot\left(\frac{h(\zt'}{4}\right);\, -\tan\left(\frac{h(\zt'}{4}\right)\right\}.
$$
\vskip0.1in

\noindent On the other hand, by our choice of $\zt :=sz_0$ and $\zt'=tz_0$ with $s, t\in\mathbb R^+$ 
we also have that
$$
\lim\limits_{\beta\to\infty} \cot\!\left(\!\frac{h(\beta\zt)}{2}\!\right) \ =\ 
\lim\limits_{\beta\to\infty} \cot\!\left(\!\frac{h(\beta\zt')}{2}\!\right).
$$
\vskip0.05in

The latter gives us a contradiction because by Lemma \ref{L:Properties-2} we have that
\vskip0.05in

$$
\left\{\cot\left(\frac{h(\zt}{4}\right);\, -\tan\left(\frac{h(\zt}{4}\right)\right\}\cap
\left\{\cot\left(\frac{h(\zt'}{4}\right);\, -\tan\left(\frac{h(\zt'}{4}\right)\right\}=\emptyset\, .
$$
\vskip0.05in
\noindent  The analysis of {\underline{\em Case 1}} is concluded.
 \vskip0.2in
\noindent  \underline{{\em Analysis of Case 2:}}\quad Suppose, for instance, that $\beta_0\zt\notin S$ for some
 $\beta_0>1$.
Then \eqref{E:A=zero-half} for $\beta=\beta_0$ reads
$$
(1-\beta_0)\cos(h(\zt)) +\beta_0-\cos(h(\zt))=0
$$
giving us
$$
\cos (h(\zt))=1,\quad \text{that is}\quad \zt\notin S\, .
$$
But this is impossible because $\zt\in S$. The situation when $\beta_0'\zt'\notin S$ for some
 $\beta_0'>1$ is treated in the same way.

 \vskip0.1in
 \noindent The proof of Theorem \ref{T:L-infty-H}
  is concluded.

 \vskip0.1in
\noindent {\bf Remark (D).} It is possible that the hypothesis that $h:\mathbb C\to\mathbb R$ is continuous may be relaxed.
Below we give an example where $h(z)$ is constant except at one point
 and we show that the corresponding $\H$ is unbounded.

Fix $\epsilon_0>0$ such that $\sin \epsilon_0>1/2$ and set
\begin{equation}\label{E:Ex-non-const}
\displaystyle{h(z) = \left\{
\begin{array}{lll}
0& \text{if}\ z\neq 0\\
-\epsilon_0 & \text{if}\ z=0
\end{array}
\right.
 } 
\end{equation}
Consider three-tuples of the form $\mathbf z_\lambda=\{0, 1, i\,\lambda\}$ with $\lambda>0$. Then
 it is easy to see that
$$
\H (\mathbf z_\lambda)\ =\ \frac12\frac{(1+\lambda^2)^{1/2}}{\lambda}\, 
\left\{\lambda\cos(\theta_{2,\,\lambda}+\epsilon_0) +\sin (\theta_{2,\,\lambda}+\epsilon_0)\right\}\, ,
$$
where $\theta_{2, \lambda}$ is the angle of $\Delta(0, 1, i\lambda)$ at the vertex
$z_2=1$.

Now the condition 
$
|\H (\mathbf z)|\,\leq C
$
for all three-tuples of non-collinear points would, in particular, require that
$$
\left\{\lambda\cos(\theta_{2,\,\lambda}+\epsilon_0) +\sin (\theta_{2,\,\lambda}+\epsilon_0)\right\}\leq\,
\frac{2\,\lambda}{(1+\lambda^2)^{1/2}}\,C\quad \text{for any}\ \lambda>0\, ,
$$
which is not possible because $\theta_{2,\,\lambda}\to 0$ as $\lambda\to 0$ and $\sin\epsilon_0>1/2$.
}


\begin{thebibliography}{199}

\bibitem{Be} Bell S. R. {\em The Cauchy integral, the Szeg\"o projection, the Dirichlet problem and the Ahlfors map,} Contemp. Math {\bf 137} (1992), 43 - 61.
\vskip0.1in
\bibitem{Calderon} Calder\'on A.-P. {\em Cauchy integrals on Lipschitz curves and related operators}, Proc. Nat. Acad. Sci. USA {\bf 74} (no. 4) (1977),  1324 - 1327.
\vskip0.1in
\bibitem{CMPT} Chousionis, V., Mateu J., Prat L. and Tolsa, X. {\em{Calder\`on-Zygmund kernels and rectifiability in the plane}} (2012), Adv. in Math., {\bf 231} (2012) 535 - 568.
\vskip0.1in
\bibitem{CMPT2} Chousionis, V., Mateu J., Prat L. and Tolsa, X. {\em{Capacities associated with Calder\'on-Zygmund kernels}} J.~Potential Anal. {\bf{38}} no.~3 (2013), 913-949.  
\vskip0.1in
\bibitem{CP} Chousionis V. and Prat L., {\em Some Calder\`on-Zygmund kernels and their relation to Wolff capacity}, Math. Z., {\bf 282} (2016) 435 - 460.
\vskip0.1in
\bibitem{C} Chunaev P. {\em{A new family of singular integral operators whose $L^2$ boundedness implies rectifiability,}} J.~Geom.~Anal. {\bf{27}} (2017) 2725-2757.   
\vskip0.1in 
\bibitem{CMT} Chunaev P., Mateu J. and Tolsa X. {\em{Singular integrals unsuitable for the curvature method whose $L^2$ boundedness still implies rectifiability,}} J. Anal. Math. {\bf{138}} (2019), 741-764. 
\vskip0.1in
\bibitem{CMT2} Chunaev P., Mateu J. and Tolsa X. {\em{A family of singular integral operators which control the Cauchy transform,}}  Math. Z., {\bf 294} (2020) 1283 - 1340.
\vskip0.1in
\bibitem{CMM} Coifman R., McIntosh A. and Meyer Y. {\em{Cauchy integrals on Lipschitz curves and related operators,}}  Proc.~Nat.~Acad.~Sci.~U.~S.~A. {\bf{74}} (1977), no.~4, 1324-1327. 
{\em L' integrale de Cauchy definit un operateur born\`e sur $L^2$ pour les courbes Lipschitziennes,} Ann. of Math. {\bf 116} (1982), 361-387.
\vskip0.1in
\bibitem{D} David G., {\em Analytic capacity, Cauchy kernel, Menger curvature, and rectifiability}, Harmonic analysis and partial differential equations (Chicago, IL, 1996), 183-197, Chicago Lectures in Math., Univ. Chicago Press, Chicago, IL, 1999. 
\vskip0.1in
\bibitem{F} Farag H., {\em The Riesz kernels do not give rise to higher dimensional analogues of the Menger-Melnikov curvature,} Publ. Mat. {\bf 43} (1999), 251-260.
\vskip0.1in
\bibitem{LPr} Lanzani L. and Pramanik M. {\em Symmetrization of a Cauchy-like kernel on curves}, to appear in J.~Funct.~Anal.
 Preprint available at: {\em{https://arxiv.org/abs/2001.09375}}.
\vskip0.1in
\bibitem{LW1} Lerman G. and Whitehouse T., {\em{ High-dimensional
      Menger-type curvatures - Part I: Geometric multipoles and
      multiscale inequalities,}} Rev. Mat. Iberoamericana {\bf{27}}
  (2011), no.~2, 493-555. 
\vskip0.1in
\bibitem{LW2} Lerman G. and Whitehouse T., {\em{High-dimensional
      Menger-type curvatures - Part II: $d$-separation and a menagerie
    of curvatures,}} Constr. Approx. {\bf 30} (2009), 325-360. 
\vskip0.1in 
\bibitem{MPV} Mateu J.,  Prat L. and Verdera J. {\em The capacity associated to signed Riesz kernels, and Wolff potentials,} J. Reine angew. Math. {\bf 578} (2005), 201-223.
\vskip0.1in 
\bibitem{MMV} Mattila P., Melnikov M. and Verdera J. {\em{The
      Cauchy integral, analytic capacity, and uniform
      rectifiability,}}  Ann. of Math. (2) 144 (1996), no. 1, 127–136. 
\vskip0.1in
\bibitem{M} Melnikov M. S., {\em Analytic capacity: a discrete approach and the curvature of measures}, Mat. Sb. {\bf 186} no. 6 (1995), 57 - 76.
\vskip0.1in
\bibitem{MV} Melnikov M. and Verdera J. {\em A geometric proof of
    the $L^2$ boundedness of the Cauchy integral on Lipschitz graphs,}
  Int. Math. Res. Not. {\bf 7} (1995), 325-331.
\vskip0.1in
\bibitem{P} Pajot, H., {\em Analytic capacity, rectifiability, Menger curvature and the Cauchy integral},
 Lecture Notes in Mathematics, 1799. Springer-Verlag, Berlin, 2002.
\vskip0.1in
\bibitem{Prat2004} Prat L., {\em{Potential theory of signed Riesz kernels: capacity and Hausdorff
measure}} Int. Math. Res. Not. {\bf 19} (2004), 937–981.
\vskip0.1in 
\bibitem{Stein-book} Stein E. {\em{Harmonic analysis:
      Real-variable methods, orthogonality and oscillatory
      integrals,}} Princeton Mathematical Series, 43. Monographs in Harmonic Analysis, III. Princeton University Press, Princeton, NJ, 1993.
\vskip0.1in
\bibitem{T} Tolsa X., {\em Analytic capacity, the Cauchy transform, andnon-homogeneous Calder\'on-Zygmud theory}, Progress in Math. {\bf 307} (2014), Birkh\"auser.
\vskip0.1in
\bibitem{Verc} Verchota G. {\em Layer potentials and regularity of the Dirichlet problem for Laplace's equation in Lipschitz domains}, J. Funct. Analysis {\bf 59} (1984), 572-611.
\vskip0.1in 
\bibitem{Verd} Verdera J., {\em $L^2$ boundedness of the Cauchy integral and Menger curvature,}
 Contemp. Math. {\bf 277} (2001), 139 -158.

\end{thebibliography}
\end{document}